\documentclass[10pt]{article}
\usepackage{amsfonts}
\usepackage{amsmath}
\usepackage{amssymb}
\usepackage{enumitem}
\allowdisplaybreaks
\usepackage[T1]{fontenc}
\DeclareMathAlphabet{\mathdutchcal}{U}{dutchcal}{m}{n}
\usepackage{cancel}
\usepackage[all]{xy}
\usepackage{mathalfa}
\usepackage[mathscr]{euscript}

\usepackage[colorlinks=true, pdfstartview=FitV, linkcolor=blue, citecolor=blue, urlcolor=blue]{hyperref}
\usepackage{xcolor}
\usepackage{setspace}
\usepackage[a4paper, left=1.2in, right=1.2in, top=1in, bottom=1in]{geometry}
\numberwithin{equation}{section}
\newtheorem{Thm}{Theorem}[section]
\newtheorem{Cor}{Corollary}[section]
\newtheorem{Def}{Definition}[section]
\newtheorem{Lemma}{Lemma}[section]
\newtheorem{Prop}{Proposition}[section]
\newtheorem{Remark}{Remark}[section]
\newtheorem{Example}{Example}[section]
\def\N{\mathbb{N}}
\def\Z{\mathbb{Z}}

\def\R{\mathbb{R}}
\def\C{\mathbb{C}}

\def\fd{\hfill$\square$}
\newcommand{\prs}[2]{\mathop{\langle#1,#2\rangle}}
\date{}
\begin{document}
\title{\emph{The Fractional Dunkl Laplacian: Extension Problem and Fundamental Solution}}
\author
{{\small\bf Chaabane REJEB}\footnote{Universit\'{e} de Sherbrooke, Sherbrooke, Qu\'{e}bec, CANADA.
Email: chaabane.rejeb@usherbrooke.ca, chaabane.rejeb@gmail.com}}
\maketitle

\begin{abstract}
Consider the Dunkl Laplacian $\Delta_k$ associated with a root system $\Phi$ in $\R^d$ and a nonnegative multiplicity function $k$ on $\Phi$. In this paper, we establish a Caffarelli-Silvestre characterization  for the fractional Dunkl Laplacian through an extension problem. We also express the corresponding fundamental solution in terms of the $\Delta_k$-Riesz kernel and prove a fractional Nash-type inequality.


\end{abstract}
\begin{center}
 \section{Introduction}
 \end{center}

The fractional Laplacian $(-\Delta)^{\alpha}$ on $\R^d$, $0<\alpha<1$,  is one of the most prominent representatives of the family of nonlocal operators. It arises in many  areas, including  diffusion phenomena, optimization, probability,  finance and others. In their seminal work \cite{Caffarelli},  Caffarelli and Silvestre gave a characterization of  the fractional Laplacian as the Dirichlet-to-Neumann map for an extension problem to the half-space $(0,+\infty)\times\R^d$. This point of view allows in somewhat to avoid the nonlocal property of the fractional Laplacian by reducing it to a local operator acting on functions defined on the cylinder $(0,+\infty)\times\R^d$.\\
In addition, the Caffarelli-Silvestre extension technique turns out to be very fruitful for recovering  an elliptic PDE approach in a nonlocal framework and  has since  been applied in numerous works (see, for instance, \cite{Caffarelli prop, Garofalo-FT, Silvestre, Stinga thesis} and  the references therein).

\medskip

The Dunkl Laplacian $\Delta_k$  is a deformation of the usual Laplace operator with a differential-difference additional terms that involve the action of a finite Coxeter group, i.e.,  a finite Euclidean reflection group, (see (\ref{laplacian}) below). Extensive  studies have been carried out on analysis associated with the Dunkl Laplacian (see for instance, \cite{dunklxu, GalRejNote, GalRej, GalRej2, GalRej3, GalRejSifi, GLR,  Rejeb25-MV, Rejeb21,  Rosler3, Rosler4,  Rosler2}). \\
It generates a  Markovian $\mathcal{C}_0$-semigroup $\left\{e^{t\Delta_k}: t\geq0\right\}$ which  can be regarded as a generalization of the usual multidimensional Gaussian semigroup (see \cite{Rosler3, Rosler-Voit, Rosler4} and Appendix A.3).
This fact together with the standard Bochner subordination principle lead to define the $\alpha$-fractional Dunkl heat semigroup, $0<\alpha<1$, whose infinitesimal generator is the operator $-(-\Delta_k)^{\alpha}$.

\medskip

The  main  purpose of this paper is to investigate  results for the fractional Dunkl Laplacian. More precisely, we provide a  heat semigroup and a pseudodifferential representation formula, derive a fundamental solution and a fractional heat kernel, establish a fractional Nash inequality and obtain a Caffarelli-Silvestre extension problem characterization.\\
We now proceed to outline the main elements of our setting. Consider a Coxeter root system $\Phi$ in $\R^d$, i.e.,  $\Phi$ is a finite subset of $\R^d\setminus\{0\}$ such that for every $\upsilon\in \Phi$, $\Phi\cap\R\upsilon=\{\pm\upsilon\}$ and $\sigma_{\upsilon}(\Phi)=\Phi$, where $\sigma_{\upsilon}$ is the Euclidean reflection with respect to the hyperplane $H_{\upsilon}$ orthogonal to $\upsilon$.  Let $W=W(\Phi)$ be the finite Coxeter group generated by the reflections $\sigma_{\upsilon}$, $\upsilon\in\Phi$.
For more detailed treatment on the topics of root systems and their reflection groups, we refer to \cite{humph,Kane}. \\
The root system $\Phi$ can be partitioned into disjoint orbits of the roots under the action of the group $W$, and  we may assign a parameter $k$, called a multiplicity, to each orbit. This defines a multiplicity function $k$ on $\Phi$.\\
Throughout this paper, the multiplicity function $k$ is assumed to be nonnegative. We also introduce the sum of multiplicities
\begin{equation}\label{gamma}
\textstyle2\gamma:=\sum_{\upsilon\in \Phi}k(\upsilon)
\end{equation}
and the weight function $\omega_k$
\begin{equation}\label{weight function}
\textstyle\omega_k(x):=\prod_{\upsilon\in \Phi}|\prs{\upsilon}{x}|^{k(\upsilon)}
\end{equation}
which is both $W$-invariant and homogeneous of degree $2\gamma$.\\
The  Dunkl Laplacian $\Delta_k$ associated with the pair $(\Phi,k)$ is defined by \cite{Dunkl1, dunklxu}\\
\begin{equation}\label{laplacian}
\Delta_kf(x)= \Delta f(x)+\sum_{\upsilon\in \Phi}k(\upsilon)\Big(\frac{\prs{\nabla f(x)}{\upsilon}}{\prs{\upsilon}{x}}-\frac{\|\upsilon\|^2}{2}.\frac{f(x)-f(\sigma_{\upsilon}(x))}{\prs{\upsilon}{x}^2}\Big),\quad f\in \mathcal{C}^2(\R^d),
\end{equation}
where $\Delta$ and $\nabla$ are the usual Laplacian and gradient, respectively.\\
According to \cite{Dunkl1}, the operator $\Delta_k$  can be expressed as the sum of the square of  Dunkl operators corresponding to the canonical basis $(e_j)_{1\leq j\leq d}$ of $\R^d$. Namely, $\Delta_k=\sum_{j=1}^dD_{e_j}^2$, where for $\xi\in\R^d$, $D_{\xi}$ is the $\xi$-directional Dunkl operator acting on $C^1$-functions by
\begin{equation}\label{Dunkl operator}
D_{\xi}f(x):=\partial_{\xi}f(x)+\sum_{\upsilon\in\Phi}k(\upsilon)\frac{\prs{\upsilon}{\xi}}{2}.\frac{f(x)-f(\sigma_\upsilon.x)}{\prs{\upsilon}{x}}.
\end{equation}
On the other hand, it is  known that  the Dunkl Laplacian is related to the usual one via the intertwining relation
\begin{equation}\label{Intertwining relation}
\Delta_kV_k=V_k\Delta,
\end{equation}
where $V_k$ is the so-called  Dunkl intertwining operator \cite{Dunkl2,dunklxu}. It is characterized as the unique linear isomorphism on the space $\mathdutchcal{P}(\R^d)$ of polynomial functions on $\R^d$ which preserves the degree of homogeneity and satisfies $V_k(1)=1$ and $D_{\xi}V_k=V_k\partial_{\xi}$, for all $\xi\in\R^d$.   In \cite{Trimeche1}, this operator $V_k$ was further extended to a topological isomorphism from the space $\mathcal{C}^{\infty}(\R^d)$ onto itself.\\
Moreover, one of the key feature of the Dunkl intertwining operator  is that it admits a Laplace integral representation of the form
\begin{equation}\label{Vk measure}
 V_k(f)(x)=\int_{\R^d}f(y)d\mu_x(y),\quad\quad \quad   f\in\mathcal{C}^{\infty}(\R^d),
\end{equation}
with a unique compactly supported probability measure $\mu_x=\mu_x^{k,W}$ on $\R^d$. This result has been established by R\"{o}sler in \cite{Rosler1}.  The support of the measure $\mu_x$ is contained in $C(x)$, the convex hull of the $W$-orbit of $x$ \cite{Dejeu} and always contains the point $x$ \cite{GalRej1}. In addition, it is shown in \cite{GalRej1} that if the multiplicity function $k$ is positive, then the support of $\mu_x$ is invariant under the  action of the Coxeter group $W$.\\
Note that, at present, there is no known general closed-form formula for the measures $\mu_x^{k,W}$,  apart from a  few special cases of Coxeter groups.\\

The paper is organized as follows. Section 2 is devoted to  examining several properties of the fractional Dunkl Laplacian via the semigroup approach. In particular,  a fundamental solution and a fractional Nash-type inequality are obtained. In Section 3,  we study the $\Delta_k$-extension problem by using a fractional Poisson kernel and  establish the Caffarelli-Silvestre harmonic characterization for the fractional Dunkl Laplacian.
To make the paper as self-contained as possible, we recall in the Appendix  some essential elements of Dunkl analysis, including the Dunkl integral transform and its kernel (called the Dunkl kernel), Dunkl translation and convolution, and the Dunkl heat kernel.

\begin{Remark}
Relationship to prior work: \emph{A preliminary version of part of the present results appeared in my HAL preprint \cite{Rejeb-21}.
Subsequently, \cite{FW} studied  some properties of the fractional Dunkl Laplacian; as noted by the authors in  the published article, my preprint predates their work and the approach taken there is different. The present  submitted version also includes also new results, such as a generalized fractional Nash inequality.}
\end{Remark}

\bigskip

 \textbf{Notations:} We summarize here some notations used frequently in the paper.
\begin{itemize}[labelsep=0.5em,left=0em]
\item For any expressions $X$ and $Y$, the notation $X\simeq Y$ (resp. $X\gtrsim Y$, $X\lesssim Y$, $X\asymp Y$)  indicates  $X=CY$ (resp. $X\geq CY$, $X\leq CY$, $C^{-1}X\leq Y\leq CY$) for some positive constant $C$ independent of significant quantities.
\item For $1\leq p\leq +\infty$, $L^p_k(\R^d)$ denotes  the Lebesgue spaces with respect to the measure $\omega_k(x)dx$. Its norm is denoted by $\|\cdot\|_{L^p_{k}(\R^d)}$.
The scalar product in the Hilbert space $L^2_k(\R^d)$ will be denoted by ${\prs{\cdot}{\cdot}}_{L^2_k(\R^d)}$ with
$$
{\prs{f}{g}}_{L^2_k(\R^d)}:=\int_{\R^d}f(x)\overline{g(x)}\omega_k(x)dx,\quad f,g\in L^2_k(\R^d).
$$
Let $L^1_{k,loc}(\R^d):=L^1_{loc}(\R^d,\omega_k(x)dx)$ be the space of locally integrable functions on $\R^d$ w.r.t. the measure $\omega_k(x)dx$.
As usual, the spaces $\mathcal{S}(\R^d)$ and $\mathcal{S}'(\R^d)$ are the Schwartz space and the space of tempered distributions, respectively.
Finally, for an open set $\Omega\subset \R^d$, $\mathcal{D}(\Omega)$ denotes $\mathcal{C}_c^{\infty}(\Omega)$ and $\mathcal{D}'(\Omega)$ its topological dual.
\end{itemize}
\section{Properties of the fractional Dunkl Laplacian}
\subsection{The fractional Dunkl Laplacian via semigroup theory}
 In this section, we focus on some properties of the fractional Dunkl Laplacian. The starting point is to define the operator $-(-\Delta_k)^{\alpha}$, $0<\alpha<1$, as the infinitesimal generator of a specific  semigroup obtained by subordination.\\
Bochner's subordination for bounded strongly continuous semigroups on Banach spaces provides a technique to construct  a new semigroup  from a given one  by using a convolution semigroup of sub-probability measures \cite{Bochner, Phillips, Schilling}. As a reference for the theory of strongly continuous semigroups on Banach spaces, we refer to
the monograph \cite{Davies}.\\
It is known that every family $\big\{\mu_t\big\}_{t\geq0}$ of vaguely continuous semigroup of sub-probability measures on $[0,+\infty)$ is naturally associated (via Laplace transform) to a unique Bernstein function and conversely \cite{Schilling} (Theorem 5.2).\\
The fractional powers $\lambda\longmapsto \lambda^{\alpha}$, $\lambda>0$, $\alpha\in (0,1)$, are among the most prominent (complete) Bernstein functions. They have the following L\'{e}vy-Kintchine representation:
\begin{equation}\label{lambda puissance}
\lambda^{\alpha}=\frac{1}{\Gamma(-\alpha)}\int_0^{\infty}\big(e^{-t\lambda}-1\big)\frac{dt}{t^{\alpha+1}},\quad\quad \quad   \forall\ \lambda>0,
\end{equation}
with $\Gamma(-\alpha):=-\Gamma(1-\alpha)/\alpha$.\\
For fixed $\alpha\in(0,1)$, let $\left\{\eta_{\alpha,t}(s)ds\right\}_{t\geq0}$ be  the unique vaguely continuous  semigroup of absolutely continuous probability measures  on $(0,+\infty)$, called the subordinator of index $\alpha$,  such that
\begin{equation}\label{subordinator laplace trans}
 e^{-t\lambda^{\alpha}}=\int_0^{\infty}e^{-s\lambda}\eta_{\alpha,t}(s)ds.
\end{equation}
According to  \cite{Rosler3, Rosler4}, we know that  the Dunkl heat semigroup $\big\{e^{t\Delta_k}:\ t\geq0\big\}$ is a strongly continuous  contraction semigroup on $X$, where $X$ is one of the Banach spaces $L^p_k(\R^d)$, with $1\leq p<+\infty$,  $\mathcal{C}_0(\R^d)$,  or $\mathcal{C}_b(\R^d)$.\\
Therefore, by the subordination principle,   formula (\ref{subordinator laplace trans}) leads to construct a new strongly continuous contraction semigroup $H_{k,\alpha}:=\big\{H_{k,\alpha}(t):\ t\geq0\big\}$ on the same Banach spaces $X$  as follows:
\begin{equation}\label{frac heat sgp}
H_{k,\alpha}(t) f:=\int_0^{\infty}e^{s\Delta_k}f\ \eta_{\alpha,t}(s)ds, \quad H_{k,\alpha}(0)f=f,\quad f\in X.
\end{equation}
We refer to $H_{k,\alpha}$ as the $\alpha$-fractional Dunkl heat semigroup. When $k\equiv 0$, this semigroup coincides with the usual fractional heat semigroup.\\
On the other hand, the kernel of the semigroup $\big\{H_{k,\alpha}(t):\ t\geq0\big\}$, called the fractional Dunkl heat kernel, is given  by
\begin{equation}\label{subrdinator kernel}
p_{k,\alpha}(t,x,y):=\int_0^{\infty}p_k(s,x,y)\eta_{\alpha,t}(s)ds, \quad\quad  x,y\in\R^d,\quad t>0.
\end{equation}
Here, $p_k(t,x,y)$ is the Dunkl heat kernel (see \cite{Rosler3, Rosler4} and Appendix A.3 for more details)
\begin{equation}\label{heat kernel 1 Intr}
p_k(t,x,y):=\frac{1}{(2t)^{d/2+\gamma}c_k}\int_{\R^d}e^{-\frac{1}{4t}\left(d(x,y,z)\right)^2}d\mu_y(z),
\end{equation}
where
\begin{equation}\label{A function}
 d(x,y,z):=\sqrt{\|x\|^2+\|y\|^2-2\prs{x}{z}},\quad \quad \|z\|\leq \|y\|,
\end{equation}
and $c_k$ is the Macdonald-Mehta constant given by
\begin{equation}\label{Mehta}
c_k:=\int_{\R^d}e^{-\frac{\|x\|^2}{2}}\omega_k(x)dx=2^{\frac{d}{2}+\gamma-1}\Gamma(\frac{d}{2}+\gamma)d_k,
\end{equation}
and  $d_k$ is the constant
\begin{equation}\label{dk}
d_k=\int_{S^{d-1}}\omega_k(\xi)d\sigma(\xi).
\end{equation}

Analogously to the classical case, the $\alpha$-fractional Dunkl heat kernel inherits some properties from the $\Delta_k$-heat kernel.
\begin{Prop} The kernel $p_{k,\alpha}(t,x,y)$ satisfies the following properties:
\begin{description}[labelsep=0.5em,left=0em]
 \item[1.] The kernel $p_{k,\alpha}$ is positive, symmetric in $x$ and $y$. Moreover, we have
$$
\|p_{k,\alpha}(t,x,\cdot)\|_{L^1_k(\R^d)}=1, \quad \quad \forall\ t>0,\  \forall\ x\in\R^d.
$$
\item[2.]For every $t>0$ and every $x\in \R^d$, the Dunkl transform of $p_{k,\alpha}(t,x,\cdot)$ is given by
\begin{equation}\label{Dunkl transf subrdinator}
\mathcal{F}_k\big(p_{k,\alpha}(t,x,\cdot)\big)(\xi)=e^{-t\|\xi\|^{2\alpha}}E_k(-ix,\xi),
\end{equation}
where $E_k$ is the Dunkl kernel (\ref{Dunkl kernel def}). In particular, $p_{k,\alpha}(t,0,\cdot)$ is a radial function.
\item[3.] The function $(t,x,y)\longmapsto p_{k,\alpha}(t,x,y)$ is of class $C^{\infty}$ on $(0,+\infty)\times \R^d\times\R^d$.
\item[4.] \emph{Scaling property:} for all $t>0$ and all $x,y\in\R^d$
\begin{equation}\label{subrdinator scaling}
p_{k,\alpha}(t,x,y)=t^{-\frac{d+2\gamma}{2\alpha}}p_{k,\alpha}(1,t^{-\frac{1}{2\alpha}}x,t^{-\frac{1}{2\alpha}}y).
\end{equation}
\item[5.] \emph{Chapman-Kolmogorov property:} for every $t,s>0$ and every $x,y\in \R^d$
\begin{equation}\label{subordinator kernel semigroup}
\int_{\R^d}p_{k,\alpha}(t,x,z)p_{k,\alpha}(s,z,y)\omega_k(z)dz=p_{k,\alpha}(t+s,x,y).
\end{equation}
\end{description}
\end{Prop}
\emph{Proof:} The first  follows immediately from the properties of the $\Delta_k$-heat kernel (see Appendix A.3). Using respectively Fubini's theorem,  the equality
 $\mathcal{F}_k\big(p_{k}(t,x,\cdot)\big)(\xi)=e^{-t\|\xi\|^{2}}E_k(-ix,\xi)$ and (\ref{subordinator laplace trans}), we obtain relation (\ref{Dunkl transf subrdinator}).\\
From the inversion formula for the Dunkl transform and  (\ref{Dunkl transf subrdinator}), we can write
\begin{equation}\label{subrdinator Def2}
p_{k,\alpha}(t,x,y)=c_k^{-2}\int_{\R^d}e^{-t\|\xi\|^{2\alpha}}E_k(-ix,\xi)E_k(iy,\xi)\omega_k(\xi)d\xi.
\end{equation}
Due to estimate (\ref{Dunkl kernel inequality2}) for the Dunkl kernel,  differentiation  under the integral sign in (\ref{subrdinator Def2}) is justified, proving  the third assertion. The scaling property follows by using the change of variables $y=t^{\frac{1}{\alpha}}\xi$ in  relation (\ref{subrdinator Def2}). As in the classical case, the  Chapman-Kolmogorov property can be proven by using the properties of the measures $\big\{\eta_{\alpha,t}(s)ds: t\geq0\big\}$. Alternatively, it can be deduced from  Plancherel's  formula for the Dunkl transform and  relations (\ref{Dunkl transf subrdinator}) and (\ref{subrdinator Def2}):
\begin{align*}
\int_{\R^d}p_{k,\alpha}(t,x,z)p_{k,\alpha}(s,z,y)\omega_k(z)dz
&=c_k^{-2}\int_{\R^d}e^{-(t+s)\|\xi\|^{2\alpha}}E_k(-ix,\xi)E_k(iy,\xi)\omega_k(\xi)d\xi\\
&=p_{k,\alpha}(t+s,x,y).
\end{align*}
\fd

\begin{Prop} For all $f\in L^2_k(\R^d)$ and all $t>0$,
\begin{equation}\label{Dunkl Trans FHS}
H_{k,\alpha}(t)f(x)=p_{k,\alpha}(t,0,\cdot)\ast_kf(x)=c_k^{-2}\int_{\R^d}e^{-t\|\xi\|^{2\alpha}}\mathcal{F}_k(f)(\xi)E_k(ix,\xi)\omega_k(\xi)d\xi,
\end{equation}
where $\ast_k$ is the Dunkl convolution product (see Appendix A.2).
\end{Prop}
\emph{Proof:} Since $p_{k,\alpha}(t,0,\cdot)$ and  $\mathcal{F}_k\big(p_{k,\alpha}(t,0,\cdot)\big)$ are both  in $L^1_k(\R^d)$, it follows from relations  (\ref{translation Schwartz}) and  (\ref{Dunkl transf subrdinator}) that $p_{k,\alpha}(t,x,\cdot)=\tau_{-x}p_{k,\alpha}(t,0,\cdot)\in  L^2_k(\R^d)$. Therefore, the first equality in (\ref{Dunkl Trans FHS}) holds. By (\ref{Dunkl trans convolution}), the second equality follows.

\begin{Def} Let $\alpha\in (0,1)$. We define the $\alpha$-power of the  Dunkl  Laplacian, denoted by $(-\Delta_k)^{\alpha}$, as the  operator
 \begin{equation}\label{Def-GI}
(-\Delta_k)^{\alpha}f:=\lim_{t\to0}\frac{1}{t}\big(f-H_{k,\alpha}(t)f\big),\quad f\in X
\end{equation}
whenever the limit exists in the Banach space $X$ (where $X$ is as above). That is, $-(-\Delta_k)^{\alpha}$ is the infinitesimal generator of the semigroup $\big\{H_{k,\alpha}(t):\ t\geq0\big\}$ on $X$.
\end{Def}
We now state some properties of the fractional Dunkl Laplacian acting on the Schwartz space $\mathcal{S}(\R^d)$.
\begin{Prop} \label{frac Lap-trans} For all  $f\in\mathcal{S}(\R^d)$, we have
\begin{equation}\label{frac Lap-Schwartz functions}
(-\Delta_k)^{\alpha}f(x)=c_k^{-2}\int_{\R^d}\|\xi\|^{2\alpha}\mathcal{F}_k(f)(\xi)E_k(ix,\xi)\omega_k(\xi)d\xi,\quad \quad x\in \R^d.
\end{equation}
In particular,
\begin{description}[labelsep=0.5em,left=0em]
\item[i)] $(-\Delta_k)^{\alpha}f$ is a $C^{\infty}$-bounded function on $\R^d$.
\item[ii)] $\lim_{\alpha\to0} (-\Delta_k)^{\alpha}f=f$ and $\lim_{\alpha\to1} (-\Delta_k)^{\alpha}f=-\Delta_kf$.
\item[iii)] The operator $(-\Delta_k)^{\alpha}$ commutes with the $W$-action, i.e.
$$
\forall\ g\in W, \quad (-\Delta_k)^{\alpha}\big(f(g\cdot)\big)=\big((-\Delta_k)^{\alpha}f\big)(g\cdot).
$$
\item[iv)] \emph{Homogeneity property:} for every $\lambda\in\R\setminus\{0\}$,
\begin{equation}\label{frac lap-dilation}
(-\Delta_k)^{\alpha}[f(\lambda\cdot)]=\lambda^{2\alpha}[(-\Delta_k)^{\alpha}f](\lambda\cdot).
\end{equation}
\item[v)] For every $\xi\in\R^d$, we have $(-\Delta_k)^{\alpha}D_{\xi}=D_{\xi}(-\Delta_k)^{\alpha}$ in $\mathcal{S}(\R^d)$.
\item[vi)] If $f\in \mathcal{S}(\R^d)$ is radial, then $(-\Delta_k)^{\alpha}f$ is also radial.
\end{description}

\end{Prop}
\emph{Proof:}
Using (\ref{Dunkl Trans FHS}) and the inversion formula for the Dunkl transform, we see that for all $t>0$ and all $x\in\R^d$,
$$
f(x)-H_{k,\alpha}(t)f(x)=c_k^{-2}\int_{\R^d}\big(1-e^{-t\|\xi\|^{2\alpha}}\big)\mathcal{F}_k(f)(\xi)E_k(ix,\xi)\omega_k(\xi)d\xi.
$$
Now, using  the inequalities $|E_k(ix,\xi)|\leq1$ and
\begin{equation}\label{Inquality Sec1}
\forall\ t>0,\ \forall\ \xi\in\R^d, \quad 0\leq t^{-1}(1-e^{-t\|\xi\|^{2\alpha}})\leq \|\xi\|^{2\alpha},
\end{equation}
we conclude that the sequence $\big\{t^{-1}(f-H_{k,\alpha}(t)f): t>0\big\}\subset \mathcal{C}_0(\R^d)$ converges uniformly to the right-hand side of (\ref{frac Lap-Schwartz functions}).\\
Statements \textbf{i)}-\textbf{v)} follow from  formula (\ref{frac Lap-Schwartz functions})  and the properties of the Dunkl transform (listed in the Appendix A.1).

\fd
\begin{Remark}\label{Remark-S} For all $\alpha> 0$, we can obviously define the power $(-\Delta_k)^{\alpha}$ on $\mathcal{S}(\R^d)$ by setting
$$
(-\Delta_k)^{\alpha}:=(-\Delta_k)^{\lfloor\alpha\rfloor}\circ (-\Delta_k)^{\alpha-\lfloor\alpha\rfloor}=(-\Delta_k)^{\alpha-\lfloor\alpha\rfloor}\circ
 (-\Delta_k)^{\lfloor\alpha\rfloor}.
$$
In particular, formula (\ref{frac Lap-Schwartz functions}) holds for every $\alpha>0$.
\end{Remark}
The next result gives a Phillips-type semigroup formula \cite{Schilling} for the fractional Dunkl Laplacian.
\begin{Prop} Let $f\in \mathcal{S}(\R^d)$ and $0<\alpha<1$. Then, we have
\begin{equation}\label{Frac Laplace-heat}
(-\Delta_k)^{\alpha}f(x)=\frac{1}{\Gamma(-\alpha)}\int_0^{\infty}\big(e^{t\Delta_k}f(x)-f(x)\big)\frac{dt}{t^{\alpha+1}}, \quad \quad x\in \R^d.
\end{equation}

\end{Prop}
\emph{Proof:} Since $e^{t\Delta_k}f\in \mathcal{S}(\R^d)$,  from  the inversion formula for the Dunkl transform we obtain
\begin{align*}
e^{t\Delta_k}f(x)-f(x)&=c_k^{-2}\int_{\R^d}\big(e^{-t\|\xi\|^2}-1\big)\mathcal{F}_k(f)(\xi)E_k(ix,\xi)\omega_k(\xi)d\xi.
\end{align*}
Combining this with the L\'{e}vy-Kintchine formula (\ref{lambda puissance}), we get
\begin{align*}
\int_0^{\infty}\big|f(x)-e^{t\Delta_k}f(x)\big|\frac{dt}{t^{\alpha+1}}
&\lesssim\int_{\R^d}\|\xi\|^{2\alpha}|\mathcal{F}_k(f)(\xi)|\omega_k(\xi)d\xi<+\infty.
\end{align*}
Thus, Fubini's theorem implies
\begin{align*}
\frac{1}{\Gamma(-\alpha)}\int_0^{\infty}\big(e^{t\Delta_k}f(x)-f(x)\big)\frac{dt}{t^{\alpha+1}}
=c_k^{-2}\int_{\R^d}\|\xi\|^{2\alpha}\mathcal{F}_k(f)(\xi)E_k(ix,\xi)\omega_k(\xi)d\xi=(-\Delta_k)^{\alpha}f(x).
\end{align*}
\fd

As an  application, we  prove the following intertwining relation.
\begin{Prop}  The following intertwining relation holds in $\mathcal{S}(\R^d)$
$$
(-\Delta_k)^{\alpha}V_k=V_k(-\Delta)^{\alpha},\quad \forall\ \alpha\geq0.
$$
\end{Prop}
\emph{Proof:} When $\alpha\in \N$, the result follows from (\ref{Intertwining relation}). For non-integer $\alpha$, it is enough to consider the case where $\alpha\in (0,1)$.\\ By the inversion formula for the usual Fourier transform $\mathcal{F}$ and Fubini's theorem, we can write
\begin{align*}
V_k(e^{t\Delta}f)(x)&=(2\pi)^{-d/2}\int_{\R^d}e^{-t\|\xi\|^2}\mathcal{F}(f)(\xi)V_k\big(e^{i\prs{\cdot}{\xi}}\big)(x)d\xi
=(2\pi)^{-d/2}\int_{\R^d}e^{-t\|\xi\|^2}\mathcal{F}(f)(\xi)E_k(ix,\xi)d\xi.
\end{align*}
Hence if we make use of this equality, the intertwining relation $\Delta_kV_k=V_k\Delta$ and the differentiation theorem under
the integral sign, then  we arrive at
$$
\Delta_kV_k(e^{t\Delta}f)=V_k\Delta(e^{t\Delta}f)=V_k(\partial_t e^{t\Delta}f)=\partial_tV_k(e^{t\Delta}f).
$$
Moreover, the dominated convergence theorem implies that  $\lim_{t\to0}V_k(e^{t\Delta}f)=V_k(f)$.
Accordingly, the function $V_k(e^{t\Delta}f)$ satisfies  the $\Delta_k$-Cauchy problem
$$
\begin{cases}
\Delta_ku(t,x)-\partial_tu(t,x)=0,\\
u(0,x)=V_k(f)\in \mathcal{C}_b(\R^d).
\end{cases}
$$
Since $V_k(f)$ is bounded, the solution of this problem is unique (see Corollary 4.4 in \cite{Rosler3}). Thus, we obtain
$$
e^{t\Delta_k}V_k(f)=V_k(e^{t\Delta}f).
$$
Combining this with (\ref{Frac Laplace-heat}) and applying  Fubini's theorem, we arrive at  the desired  intertwining relation.

\fd
\begin{Remark} By means of the Fourier and  Dunkl transforms,  we also obtain another intertwining relation:
$$
(\mathcal{F}^{-1}\circ \mathcal{F}_k)(-\Delta_k)^{\alpha}=(-\Delta)^{\alpha}(\mathcal{F}^{-1}\circ \mathcal{F}_k)\quad \text{in}\quad \mathcal{S}(\R^d).
$$
\end{Remark}

\begin{Prop} The domain of  $(-\Delta_k)^{\alpha}$ in $L^2_k(\R^d)$ is the generalized Sobolev space
$$
H_k^{2\alpha}(\R^d):=\Big\{f\in L^2_k(\R^d): \quad \|\cdot\|^{2\alpha}\mathcal{F}_k(f)\in L^2_k(\R^d)\Big\}.
$$
In addition, for every $f\in H_k^{2\alpha}(\R^d)$ we have
\begin{equation}\label{generateur infinitesimal L2}
\mathcal{F}_k\big((-\Delta_k)^{\alpha}f\big)=\|\cdot\|^{2\alpha}\mathcal{F}_k(f)\quad \quad  \text{in}\quad L^2_k(\R^d)
\end{equation}
and
\begin{equation}\label{norm Frac Lap}
\big\|(-\Delta_k)^{\alpha}f\big\|_{L^2_k(\R^d)}=c_k^{-1}\left\|\|\cdot\|^{2\alpha}\mathcal{F}_k(f)\right\|_{L^2_k(\R^d)}.
\end{equation}
\end{Prop}
\emph{Proof:} Let $\mathfrak{D}_{\alpha,2}$ be the domain of $(-\Delta_k)^{\alpha}$ in $L^2_k(\R^d)$ and $f\in\mathfrak{D}_{\alpha,2}$.  Using Plancherel's theorem and (\ref{Dunkl Trans FHS}), we obtain  for a.e. $\xi\in \R^d$
\begin{align*}
\mathcal{F}_k\big((-\Delta_k)^{\alpha}f\big)(\xi)&=\lim_{t\to0}t^{-1}\Big(\mathcal{F}_k(f)(\xi)-\mathcal{F}_k\big(H_{k,\alpha}(t)f\big)(\xi)\Big)\\
&=\lim_{t\to0}t^{-1}\big(1-e^{-t\|\xi\|^{2\alpha}}\big)\mathcal{F}_k(f)(\xi)\\
&=\|\xi\|^{2\alpha}\mathcal{F}_k(f)(\xi).
\end{align*}
Thus, $\|\cdot\|^{2\alpha}\mathcal{F}_k(f)\in L^2_k(\R^d)$ and $\mathfrak{D}_{\alpha,2}\subset H^{2\alpha}_k(\R^d)$.
Conversely, if $f\in H_k^{2\alpha}(\R^d)$, then inequality (\ref{Inquality Sec1}) implies that the dominated convergence theorem applies, yielding
$$
\lim_{t\rightarrow0}t^{-1}\Big(\mathcal{F}_k(f)-\mathcal{F}_k\big(H_{k,\alpha}(t)f\big)\Big)=\|\cdot\|^{2\alpha}\mathcal{F}_k(f)\quad \text{in}\quad L^2_k(\R^d).
$$
Hence, by Plancherel's theorem, we deduce that $H_k^{2\alpha}(\R^d)\subset \mathfrak{D}_{\alpha,2}$. On the other hand, using (\ref{generateur infinitesimal L2}) and applying Plancherel's formula for the Dunkl transform $\mathcal{F}_k$, we obtain  equality (\ref{norm Frac Lap}).

\fd

\medskip

By virtue of (\ref{Dunkl transf subrdinator}), we see that for all $t>0$ and all $x\in\R^d$, the function $p_{k,\alpha}(t,x,\cdot)$ belongs to the generalized Sobolev space
$H_{k}^{2\alpha}(\R^d)$.  The next theorem shows that the kernel $p_{k,\alpha}(t,\cdot,\cdot)$ is indeed the fundamental solution of the fractional Dunkl heat operator $\partial_t+(-\Delta_k)^{\alpha}$.
\begin{Thm} The kernel  $p_{k,\alpha}(t,\cdot,\cdot)$ is the fractional heat kernel associated with $(-\Delta_k)^{\alpha}$. That is,  for every $x\in\R^d$
\begin{equation*}\label{Frac heat kernel}
\begin{cases}
\left(\partial_t+(-\Delta_k)^{\alpha}\right)p_{k,\alpha}(t,x,\cdot)=0, & \hbox{in}\quad  \R^d \\
\lim_{t\to0}p_{k,\alpha}(t,x,\cdot)\omega_k=\delta_x, & \hbox{in}\quad  \mathcal{S}'(\R^d).
\end{cases}
\end{equation*}
Moreover,  if $f\in H^{2\alpha}_k(\R^d)$, then the function $e^{-t(-\Delta_k)^{\alpha}}f:=H_{k,\alpha}(t)f$ solves the  fractional heat equation
\begin{equation*}\label{Frac heat kernel}
\big(\partial_t+(-\Delta_k)^{\alpha}\big)u(t,\cdot)=0\quad \text{and} \quad \lim_{t\to0}u(t,\cdot)=f.
\end{equation*}
\end{Thm}
\emph{Proof:} Using respectively the relations (\ref{generateur infinitesimal L2}), (\ref{Dunkl transf subrdinator})
and (\ref{subrdinator Def2}), we deduce that
\begin{align*}
\mathcal{F}_k\big((-\Delta_k)^{\alpha}p_{k,\alpha}(t,x,\cdot)\big)&=\|\cdot\|^{\alpha}E_k(ix,\cdot)e^{-t\|\cdot\|^{2\alpha}}
=-\partial_t\mathcal{F}_k\big(p_{k,\alpha}(t,x,\cdot)\big)=-\mathcal{F}_k\big(\partial_tp_{k,\alpha}(t,x,\cdot)\big).
\end{align*}
Thus, by the injectivity of the Dunkl transform implies that
$$
(-\Delta_k)^{\alpha}p_{k,\alpha}(t,x,\cdot)+\partial_tp_{k,\alpha}(t,x,\cdot)=0.
$$
On the other hand, from (\ref{Dunkl Trans FHS}) we get
\begin{align*}
 \lim_{t\to0}\prs{\mathcal{F}_k\big(p_{k,\alpha}(t,x,\cdot)\omega_k\big)}{\varphi}
=\lim_{t\to0}H_{k,\alpha}(t)\mathcal{F}_k(\phi)(x)=\mathcal{F}_k(\varphi)(x)=\prs{\mathcal{F}_k(\delta_x)}{\varphi}, \quad \forall\ \varphi\in\mathcal{S}(\R^d).
\end{align*}
Since the Dunkl transform $\mathcal{F}_k:\mathcal{S}'(\R^d)\longrightarrow\mathcal{S}'(\R^d)$ is an isomorphism, the first claim follows.\\
The second part is a direct consequence of (\ref{Dunkl Trans FHS}) and (\ref{generateur infinitesimal L2}).

\fd

\medskip

We collect additional properties of the fractional Dunkl Laplacian. They can be obtained directly from (\ref{generateur infinitesimal L2}) and the properties of the Dunkl transform, so we omit the details.
\begin{Prop}Let $f,g\in H_k^{2\alpha}(\R^d)$.
\begin{description}
  \item[i)] Translation invariance: for all $x\in \R^d$,
\begin{equation}\label{frac lap-translation}
\tau_x(-\Delta_k)^{\alpha}f=(-\Delta_k)^{\alpha}\tau_xf, \quad \text{in}\quad  L^2_k(\R^d).
\end{equation}
 \item[ii)] Convolution invariance:
\begin{equation}\label{frac lap-convolution}
(-\Delta_k)^{\alpha}(f\ast_kg)=\big((-\Delta_k)^{\alpha}f\big)\ast_kg=f\ast_k\big((-\Delta_k)^{\alpha}g\big),\quad \text{in}\quad  L^2_k(\R^d).
\end{equation}
\item[iii)] Symmetry:
\begin{equation}\label{symmetry}
{\prs{(-\Delta_k)^{\alpha}f}{g}}_{L^2_k(\R^d)}={\prs{(-\Delta_k)^{\alpha/2}f}{(-\Delta_k)^{\alpha/2}g}}_{L^2_k(\R^d)}
={\prs{f}{(-\Delta_k)^{\alpha}g}}_{L^2_k(\R^d)}.
\end{equation}
\end{description}

\end{Prop}

In the next result, we calculate  the image of the  Dunkl kernel $E_k(i\xi,\cdot)\in \mathcal{C}_b(\R^d)$ (see Appendix A.1)  under  $(-\Delta_k)^{\alpha}$.
\begin{Prop} For any fixed $\xi\in \R^d$, we have
\begin{equation}\label{Dunkl kerenl-Lap}
(-\Delta_k)^{\alpha}E_k(i\xi,\cdot)=\|\xi\|^{2\alpha}E_k(i\xi,\cdot).
\end{equation}
Let $J_k$ be the generalized Bessel function defined by \cite{Opdam}:
$$
J_k(x,y):=\frac{1}{|W|}\sum_{g\in W}E_k(gx,y),\qquad g\in W.
$$
Then for every $\xi\in \R^d$
\begin{equation}\label{Opdam-Lap}
(-\Delta_k)^{\alpha}J_k(i\xi,\cdot)=\|\xi\|^{2\alpha}J_k(i\xi,\cdot).
\end{equation}
\end{Prop}
\emph{Proof:} Let $\xi\in \R^d$. By virtue of relation (\ref{Dunkl transf subrdinator}), we see that
$$
\forall\ t>0, \quad H_{k,\alpha}(t)\big(E_k(i\xi,\cdot)\big)(x)=\mathcal{F}_k\big(p_{k,\alpha}(t,x,\cdot)\big)(\xi)=e^{-t\|\xi\|^{2\alpha}}E_k(i\xi,x).
$$
Accordingly, the sequence of continuous bounded functions
$$
\Big\{t^{-1}\big(E_k(i\xi,\cdot)-H_{k,\alpha}(t)\big(E_k(i\xi,\cdot)\big)\Big\}_{t>0}
$$
converges uniformly to $\|\xi\|^{2\alpha}E_k(i\xi,\cdot)$ as $t\to0$. This gives (\ref{Dunkl kerenl-Lap}).

\fd

\begin{Example}
\end{Example}
In the rank one the root system is $\Phi=\{\pm1\}$, the reflection group is $W=\Z_2$ and the multiplicity function is a parameter $k> 0$.
By using (\ref{laplacian}), we deduce that the action of the $\Z_2$-Dunkl Laplacian is given by
$$
\Delta_k^{\Z_2}f(x)=f''(x)+2k\frac{f'(x)}{x}-k\frac{f(x)-f(-x)}{x^2}.
$$
Moreover, the $\Z_2$-Dunkl kernel can be written as \cite{Rosler4}
$$
E_k^{\Z_2}(i\xi,x)=j_{k-1/2}(x\xi)+\frac{ix\xi}{2k+1}j_{k+1/2}(x\xi),
$$

where $j_{\nu}$, $\nu>-1/2$, is the normalized Bessel function \cite{Watson}
\begin{equation*}
j_{\nu}(\lambda):=\Gamma(\nu+1)\sum_{n=0}^{\infty}\frac{(-1)^n\big(\lambda/2\big)^{2n}}{n!\Gamma(n+\nu+1)}.
\end{equation*}
In particular, the generalized $\Z_2$-Bessel function (\ref{Opdam-Lap}) is given by
$$
J_k^{\Z_2}(i\xi,x)=j_{k-1/2}(x\xi)
$$
and (\ref{Opdam-Lap}) implies that
$$
(-\Delta_k^{\Z_2})^{\alpha}j_{k-1/2}(i\xi,\cdot)=|\xi|^{2\alpha}j_{k-1/2}(i\xi,\cdot).
$$
\subsection{Fundamental solution}
In  this section, we are concerned with  the fundamental solution  of the fractional Dunkl Laplacian $(-\Delta_k)^{\alpha/2}$, for all $\alpha\in (0,d+2\gamma)$. The case where $\frac{\alpha}{2}$ is a positive integer (i.e., the polyharmonic Dunkl Lapalcian) was examined  in \cite{GalRejSifi}.\\
From \cite{GalRejSifi},  we recall that  the $\Delta_k$-Riesz kernel of index $\theta\in (0,d+2\gamma)$ is  defined by
\begin{equation}\label{Riesz kernel def}
R_{k,\theta}(x,y):=\frac{1}{\Gamma(\frac{\theta}{2})}\int_{0}^{\infty}t^{\frac{\theta}{2}-1}p_k(t,x,y)dt,
\end{equation}
with $p_k(t,x,y)$ being the $\Delta_k$-heat kernel (\ref{heat kernel 1 Intr}).
According to \cite{GalRejSifi},   relation (\ref{Riesz kernel def}) can be rewritten as
\begin{equation}\label{Riesz kernel def1}
R_{k,\theta}(x,y)=A_{k,d,\theta}\int_{\R^d}\big(d(x,y,z)\big)^{\theta-d-2\gamma}d\mu_y(z),
\end{equation}
where $\mu_y$ is the measure (\ref{Vk measure}) and $A_{k,d,\theta}$ is the  positive constant, expressed in   terms of the constants $c_k$ from (\ref{Mehta}) and $d_k$ from (\ref{dk}), given by
\begin{equation}\label{kappa}
A_{k,d,\theta}=\frac{2^{\frac{d}{2}+\gamma-\theta}\Gamma(\frac{d+2\gamma-\theta}{2})}{c_k\Gamma(\theta/2)}
=\frac{2^{1-\theta}\Gamma(\frac{d+2\gamma-\theta}{2})}{\Gamma(d/2+\gamma)\Gamma(\theta/2)d_k}.
\end{equation}
Observe that if $x=0$, then (\ref{Riesz kernel def1}) reduces to $R_{k,\theta}(0,\cdot)=A_{k,d,\theta}\|\cdot\|^{\theta-d-2\gamma}$.  \\
Moreover, it was shown in \cite{GalRejSifi} that for each $x\in \R^d$, the function  $R_{k,\theta}(x,\cdot)\omega_k$ is locally integrable w.r.t. the Lebesgue measure on $\R^d$ and defines a tempered distribution.\\
The $\Delta_k$-Riesz potential of a nonnegative Radon measure $\mu$ satisfying the following decay:
\begin{equation}\label{condition Riesz}
\int_{\R^d}\frac{d\mu(y)}{\big(1+\|y\|\big)^{d+2\gamma-\theta}}\asymp \int_{\R^d}\frac{d\mu(y)}{1+\|y\|^{d+2\gamma-\theta}}<+\infty
\end{equation}
is defined by
\begin{equation}\label{Riesz pot def}
\mathcal{I}_{k,\theta}[\mu](x):=\int_{\R^d}R_{k,\theta}(x,y)d\mu(y).
\end{equation}
Note that the above integrability condition on the measure $\mu$ implies that $\mathcal{I}_{k,\theta}[\mu]\in L^1_{k,loc}(\R^d)$. When $\mu=f\omega_k$,  we simplify notation by writing $\mathcal{I}_{k,\theta}[f]$ instead of $\mathcal{I}_{k,\theta}[f\omega_k]$. For further details on the potential theory of the $\Delta_k$-Riesz kernel, we refer the reader to \cite{GalRejSifi}.\\
Let $u\in L^1_{k,loc}(\R^d)=L^1_{loc}(\R^d,\omega_k(x)dx)$. In particular,  the function $u\omega_k$ defines a distribution. If the linear functional
$$
\mathcal{D}(\R^d)\ni f\longmapsto \prs{u\omega_k}{(-\Delta_k)^{\alpha}f}:=\int_{\R^d}u(x)(-\Delta_k)^{\alpha}f(x)\omega_k(x)dx
$$
defines a distribution, then the symmetry property (\ref{symmetry}) of the nonlocal operator $(-\Delta_k)^{\alpha}$ enables us to define the weak fractional Dunkl Laplacian of the distribution $u\omega_k$  by assuming that
\begin{equation}\label{weak frac Lap}
\prs{(-\Delta_k)^{\alpha}(u\omega_k)}{f}:=\prs{u\omega_k}{(-\Delta_k)^{\alpha}f}.
\end{equation}

\begin{Thm}\label{fund solution thm} Let $\alpha\geq0$ and $0<\theta<d+2\gamma$.
\begin{description}[labelsep=0.5em,left=0em]
\item[i)] For every $f\in \mathcal{S}(\R^d)$, we have
\begin{equation}\label{frac lap-Riesz poten}
\mathcal{I}_{k,\theta}[(-\Delta_k)^{\frac{\alpha}{2}}f]=
\begin{cases}
\mathcal{I}_{k,\theta-\alpha}[f], & \hbox{if}\quad \theta>\alpha, \\
(-\Delta_k)^{\frac{\alpha-\theta}{2}}f, & \hbox{if}\quad  \theta\leq\alpha.
\end{cases}
\end{equation}
\item[ii)] For each $x\in \R^d$, the linear functional
$$
f\longmapsto \prs{(-\Delta_k)^{\frac{\alpha}{2}}\big(R_{k,\theta}(x,\cdot)\omega_k\big)}{f}:=\mathcal{I}_{k,\theta}[(-\Delta_k)^{\frac{\alpha}{2}}f](x)
$$
defines a tempered distribution. Moreover, in $\mathcal{S}'(\R^d)$, we have
\begin{equation}\label{fundamental solution}
(-\Delta_k)^{\frac{\alpha}{2}}\big(R_{k,\theta}(x,\cdot)\omega_k\big)=
\begin{cases}
R_{k,\theta-\alpha}(x,\cdot)\omega_k, & \hbox{if}\quad \theta>\alpha, \\
(-\Delta_k)^{\frac{\alpha-\theta}{2}}\delta_x, & \hbox{if}\quad  \theta\leq\alpha.
\end{cases}
\end{equation}
In particular, for each $\alpha\in (0,d+2\gamma)$, the locally integrable function  $R_{k,\alpha}(x,\cdot)\omega_k$ is the fundamental solution of the fractional Dunkl Laplacian of power $\frac{\alpha}{2}$.
\item[iii)]Let $f\in L^1_k(\R^d)\cap L^{\infty}_k(\R^d)$. Then the function $u=\mathcal{I}_{k,\alpha}[f]$ satisfies the $(-\Delta_k)^{\frac{\alpha}{2}}$-Poisson equation
$$
(-\Delta_k)^{\frac{\alpha}{2}}(u\omega_k)=f\omega_k\quad \text{in}\quad \mathcal{S}'(\R^d).
$$
\end{description}

\end{Thm}
We need the following two lemmas.
\begin{Lemma}
If $f$ and its Dunkl transform $\mathcal{F}_k(f)$ are both in $L^1_k(\R^d)$, then for all $x\in \R^d$,
\begin{equation}\label{potential-heat sem2}
\mathcal{I}_{k,\theta}[f](x)=\frac{1}{\Gamma(\theta/2)}\int_0^{\infty}e^{t\Delta_k}f(x)t^{\frac{\theta}{2}-1}dt
=c_k^{-2}\int_{\R^d}\|\xi\|^{-\theta}\mathcal{F}_k(f)(\xi)E_k(ix,\xi)\omega_k(\xi)d\xi.
\end{equation}
Moreover,  we have
\begin{equation}\label{inequality fund solution}
\|\mathcal{I}_{k,\theta}[f]\|_{\infty}\leq \|\mathcal{I}_{k,\theta}[|f|]\|_{\infty}\lesssim \|f\|_{\infty}+\|f\|_{L^1_k(\R^d)}.
\end{equation}
\end{Lemma}
\emph{Proof:} From (\ref{Riesz kernel def}), we can  write
$$
\mathcal{I}_{k,\theta}[|f|](x)=\frac{1}{\Gamma(\theta/2)}\int_0^{\infty}e^{t\Delta_k}|f|(x)t^{\frac{\theta}{2}-1}dt
=\frac{1}{\Gamma(\theta/2)}\int_0^{1}\quad+\quad\frac{1}{\Gamma(\theta/2)} \int_1^{\infty}
$$
Noting  that $f$ is bounded  and using the inequality  $0\leq p_k(t,x,y)\leq c_k^{-1}(2t)^{-\frac{d}{2}-\gamma}$ (see relation (\ref{heat kernel 1 Intr}))  we deduce that
$$
\forall\ x\in\R^d, \quad \mathcal{I}_{k,\theta}[|f|](x)\lesssim \|f\|_{\infty}+\|f\|_{L^1_k(\R^d)}.
$$
Therefore Fubini's theorem is applied to get the first equality in (\ref{potential-heat sem2}). Moreover, we obtain the second equality by using the relation
$$
e^{t\Delta_k}f(x)= c_k^{-2}\int_{\R^d}e^{-t\|\xi\|^2}\mathcal{F}_k(f)(\xi)E_k(ix,\xi)\omega_k(\xi)d\xi
$$
together with Fubini's theorem.

\fd

\begin{Lemma}  For every $\theta>0$ and every $1\leq p\leq +\infty$, there exists  a positive constant $C=C(p,\theta)$ such that
\begin{equation}\label{norm frac Lapl Lp}
\big\|(-\Delta_k)^{\theta}f\big\|_{L^p_k(\R^d)}\leq C \left(\big\|\Delta_k^{\lfloor\theta\rfloor}f\big\|_{L^p_k(\R^d)}+\big\|\Delta_k^{\lfloor\theta\rfloor+1}f\big\|_{L^p_k(\R^d)}\right),\quad \forall\
f\in \mathcal{S}(\R^d).
\end{equation}

\end{Lemma}
\emph{Proof:} The result holds true when $\theta$ is a positive integer. Suppose  $\theta\notin\N$. By replacing $f$ with $\Delta_k^{\lfloor\theta\rfloor}f\in \mathcal{S}(\R^d)$, it suffices to consider  $\theta\in (0,1)$. From the semigroup formula (\ref{Frac Laplace-heat}), let us write
\begin{align*}
\Gamma(-\theta)(-\Delta_k)^{\theta}f(x)&=\int_0^{1}\Big(e^{t\Delta_k}f(x)-f(x)\Big)\frac{dt}{t^{\theta+1}}
+\int_1^{\infty}\Big(e^{t\Delta_k}f(x)-f(x)\Big)\frac{dt}{t^{\theta+1}}\\
&=A[f](x)+B[f](x).
\end{align*}
When  $p=1$ or $p=\infty$ we have
$$
\|B[f]\|_{L_k^p(\R^d)}\leq \int_1^{\infty}\Big(\|e^{t\Delta_k}f\|_{L^p_k(\R^d)}+\|f\|_{L^p_k(\R^d)}\Big)\frac{dt}{t^{\theta+1}}\leq \frac{2}{\theta}\|f\|_{L^p_k(\R^d)}.
$$
When $1<p<+\infty$, then the Jensen inequality combined with the $L^p_k(\R^d)$-boundedness property  of the semigroup $\big\{e^{t\Delta_k}\big\}_t$ yields  that
$$
\|B[f]\|^p_{L_k^p(\R^d)}\leq \theta^{p-1}\int_1^{\infty}\Big(\|e^{t\Delta_k}f\|^p_{L^p_k(\R^d)}+\|f\|^p_{L^p_k(\R^d)}\Big)\frac{dt}{t^{\theta+1}}\leq 2\theta^{p-2}\|f\|_{L^p_k(\R^d)}.
$$
Furthermore, since $\partial_se^{s\Delta_k}=e^{s\Delta_k}\Delta_k$, it follows that
\begin{align*}
\big|A[f](x)\big|&=\bigg|\int_0^{1}\bigg(\int_0^te^{s\Delta_k}\Delta_kf(x)ds\bigg)\frac{dt}{t^{\theta+1}}\bigg|
\leq\frac{1}{\theta}\int_0^{1}e^{s\Delta_k}|\Delta_kf|(x) (s^{-\theta}-1)ds\\
&\leq \frac{1}{\theta}\int_0^{1}e^{s\Delta_k}|\Delta_kf|(x) s^{-\theta}ds
\end{align*}
Therefore, as above, obtain
$$
\begin{cases}
\|A[f]\|_{L^p_k(\R^d)}\leq \theta^{-1}(1-\theta)^{-1}\|\Delta_kf\|_{L^p_k(\R^d)}, \hspace{0.5cm}& \hbox{if}\quad p=1\ \text{or}\ p=\infty;\\
\|A[f]\|^p_{L^p_k(\R^d)}\leq(1-\theta)^{p-2}\theta^{-p}\|\Delta_kf\|^p_{L^p_k(\R^d)}, \hspace{0.5cm}& \hbox{if}\quad 1<p<+\infty.
\end{cases}
$$
\fd

\medskip

\noindent \emph{Proof of Theorem \ref{fund solution thm}:}
\textbf{i)} Take $f\in \mathcal{S}(\R^d)$. Using  inequality (\ref{norm frac Lapl Lp}) and  the fact that $\mathcal{F}_k\big((-\Delta_k)^{\frac{\alpha}{2}}f\big)=\|\cdot\|^{\alpha}\mathcal{F}_k(f)\in L^1_k(\R^d)$,  we can replace $f$ by $(-\Delta_k)^{\frac{\alpha}{2}}f$ in (\ref{potential-heat sem2}) to write
$$
\mathcal{I}_{k,\theta}[(-\Delta_k)^{\frac{\alpha}{2}}f](x)=c_k^{-2}\int_{\R^d}\|\xi\|^{\alpha-\theta}\mathcal{F}_k(f)(\xi)E_k(ix,\xi)\omega_k(\xi)d\xi.
$$
Therefore, applying (\ref{potential-heat sem2})  when $\theta>\alpha$, the inversion formula for the Dunkl transform when $\theta=\alpha$, and Remark \ref{Remark-S} when $\theta<\alpha$, we get the desired result.

\medskip

\noindent\textbf{ii)} Because of (\ref{frac lap-Riesz poten}), it is enough to check  that for
each $x\in \R^d$, the linear functional
$$
f\longmapsto \prs{R_{k,\theta}(x,\cdot)\omega_k}{(-\Delta_k)^{\frac{\alpha}{2}}f}=\mathcal{I}_{k,\theta}[(-\Delta_k)^{\frac{\alpha}{2}}f](x)
$$
defines a tempered distribution.\\
From inequalities (\ref{inequality fund solution}) and (\ref{norm frac Lapl Lp}), we can find a positive constant
$C=C(k,d,\alpha,\theta)$ such that
\begin{align*}
\big|\mathcal{I}_{k,\theta}[(-\Delta_k)^{\frac{\alpha}{2}}f](x)\big|&\leq C \big(\|\Delta_k^{m}f\|_{\infty}+\|\Delta_k^{m+1}f\|_{\infty}
+\|\Delta_k^{m}f\|_{L^1_k(\R^d)}+\|\Delta_k^{m+1}f\|_{L^1_k(\R^d)}\big),
\end{align*}
with $m=\lfloor\frac{\alpha}{2}\rfloor$.\\
Accordingly,  this inequality together with the continuity of the polyharmonic Dunkl Laplacian  $\Delta_k^m=c_k^{-2}\mathcal{F}_k^{-1}\big(\|.\|^{2m}\mathcal{F}_k\big)$ on the Schwartz space,  implies  that $(-\Delta_k)^{\frac{\alpha}{2}}\big(R_{k,\theta}(x,\cdot)\omega_k\big)\in \mathcal{S}'(\R^d)$.\\
Finally, note that   inequality (\ref{norm frac Lapl Lp}) (for $p=\infty$) shows  that for any $\alpha>0$ and any $x\in \R^d$, the linear functional
$$
\prs{(-\Delta_k)^{\alpha}\delta_x}{f}:=(-\Delta_k)^{\alpha}f(x),\quad \quad \quad \forall\ f\in \mathcal{S}(\R^d),
$$
is also a tempered distribution.

\medskip

\noindent \textbf{iii)} For $\psi\in \mathcal{S}(\R^d)$, consider the linear functional
$$
\prs{T}{\psi}:=\prs{(-\Delta_k)^{\frac{\alpha}{2}}\big(\mathcal{I}_{k,\alpha}[f]\omega_k\big)}{\psi}
:=\int_{\R^d}\mathcal{I}_{k,\alpha}[f](x)(-\Delta_k)^{\frac{\alpha}{2}}\psi(x)\omega_k(x)dx.
$$
By virtue of  inequalities (\ref{inequality fund solution}) and (\ref{norm frac Lapl Lp}), we see that $T$ satisfies
\begin{align*}
|\prs{T}{\psi}|&\leq \|\mathcal{I}_{k,\alpha}[f]\|_{\infty} \|(-\Delta_k)^{\frac{\alpha}{2}}\psi\|_{L^1_k(\R^d)}\\
&\leq C(\alpha)\|\mathcal{I}_{k,\alpha}[f]\|_{\infty} \big(\|\Delta_k^{m}\psi\|_{L^1_k(\R^d)}+\|\Delta_k^{m+1}\psi\|_{L^1_k(\R^d)}\big).
\end{align*}
Hence $T$ is a tempered distribution. In addition, by Fubini's theorem and (\ref{frac lap-Riesz poten}) we have
$$
\prs{T}{\psi}=\int_{\R^d}\mathcal{I}_{k,\alpha}[(-\Delta_k)^{\frac{\alpha}{2}}\psi](y)f(y)\omega_k(y)dy=\prs{f\omega_k}{\psi}.
$$

\fd

In the next result, we  express  the Riesz kernel (resp. potential) in terms of the fractional heat kernel (resp. semigroup).
\begin{Prop} Let $0<\alpha<1$ and $0<\theta<d+2\gamma$. Then
\begin{equation}\label{Riesz-frac heat}
R_{k,\theta}(x,y)=\frac{1}{\Gamma(\frac{\theta}{2\alpha})}\int_0^{\infty}t^{\frac{\theta}{2\alpha}-1}p_{k,\alpha}(t,x,y)dt,\quad x,y\in \R^d.
\end{equation}
If $f\in L^1_{k,-\frac{\theta}{2}}(\R^d)$, i.e. $f\in L^{1}_{k,loc}(\R^d)$ satisfies  (\ref{condition Riesz}), and $f$ is nonnegative, then
$$
\mathcal{I}_{k,\theta}[f]=\frac{1}{\Gamma(\frac{\theta}{2\alpha})}\int_0^{\infty}t^{\frac{\theta}{2\alpha}-1}e^{-t(-\Delta_k)^{\alpha}}fdt.
$$
\end{Prop}
\emph{Proof:}  It suffices  to justify formula (\ref{Riesz-frac heat}). According to formula (23) in \cite{Bogdan}, the subordinator $\{\eta_{\alpha.t}(s)ds\}_t$
(\ref{subordinator laplace trans}) satisfies
$$
\forall\ \theta>0, \quad \int_0^{\infty}t^{\frac{\theta}{2\alpha}-1}\eta_{\alpha,t}(s)dt=\frac{\Gamma(\frac{\theta}{2\alpha})}{\Gamma(\theta/2)}s^{\frac{\theta}{2}-1}.
$$
Thus, using this identity   together with (\ref{Riesz kernel def}), we see that (\ref{Riesz-frac heat}) holds.

\fd

\subsection{Generalized fractional Nash inequality}
The next result states a fractional Nash-type inequality involving the fractional Dunkl Laplacian. Note that a Nash inequality involving Dunkl operators was recently worked out in  \cite{Velicu}.
\begin{Prop} Let $\alpha>0$. There exists a positive constant $C=C(d,k,\alpha)$ such that for every $f\in H_k^{\alpha}(\R^d)\cap L^1_k(\R^d)$,
\begin{equation}\label{Nash inequality}
\|f\|_{L^2_k(\R^d)}^{1+\frac{2\alpha}{d+2\gamma}}\leq C\|(-\Delta_k)^{\frac{\alpha}{2}}f\|_{L^2_k(\R^d)}\|f\|_{L^1_k(\R^d)}^{\frac{2\alpha}{d+2\gamma}}.
\end{equation}
\end{Prop}
\noindent\emph{Proof:} The proof is an adaptation of the approach used in \cite{Nash, Velicu}.
Let $R>0$. We have
$$
\int_{\|\xi\|>R}|\mathcal{F}_k(f)(\xi)|^2\omega_k(\xi)d\xi\leq \frac{1}{R^{2\alpha}}\int_{\|\xi\|>R}\|\xi\|^{2\alpha}|\mathcal{F}_k(f)(\xi)|^2\omega_k(\xi)d\xi\leq \frac{c_k^2}{R^{2\alpha}}\|(-\Delta_k)^{\frac{\alpha}{2}}f\|^2_{L^2_k(\R^d)},
$$
where $c_k$ is the constant defined in (\ref{Mehta}).
On the other hand, since $f$ belongs to  $L^1_k(\R^d)$, we have  $\|\mathcal{F}_k(f)\|_{\infty}\leq \|f\|_{L^1_k(\R^d)}$ and, as a consequence,
$$
\int_{\|\xi\|\leq R}|\mathcal{F}_k(f)(\xi)|^2\omega_k(\xi)d\xi\leq
\left(\int_{\|\xi\|\leq R}\omega_k(\xi)d\xi\right)\|f\|_{L^1_k(\R^d)}^2=\frac{d_kR^{d+2\gamma}}{d+2\gamma}\|f\|_{L^1_k(\R^d)}^2.
$$
with $d_k$ being the normalization constant in (\ref{dk}).
Therefore, for every $R>0$, we obtain
\begin{align*}
\forall\ R>0,\quad \|\mathcal{F}_k(f)\|_{L^2_k(\R^d)}^2&\leq \frac{d_kR^{d+2\gamma}}{d+2\gamma}\|f\|_{L^1_k(\R^d)}^2+\frac{c_k^2}{R^{2\alpha}}\|(-\Delta_k)^{\frac{\alpha}{2}}f\|^2_{L^2_k(\R^d)}.
\end{align*}
Next, observing that the infimum of the right-hand side is attained at
$$
R=\left(\frac{2\alpha c_k^2 \|(-\Delta_k)^{\frac{\alpha}{2}}f\|^2_{L^2_k(\R^d)}}{d_k\|f\|_{L^1_k(\R^d)}^2}\right)^{\frac{1}{d+2\gamma+2\alpha}},
$$
we conclude that
$$
\|\mathcal{F}_k(f)\|_{L^2_k(\R^d)}\lesssim \|f\|_{L^1_k(\R^d)}^{\frac{2\alpha}{d+2\gamma+2\alpha}}\|(-\Delta_k)^{\frac{\alpha}{2}}f\|^{\frac{d+2\gamma}{d+2\gamma+2\alpha}}_{L^2_k(\R^d)}.
$$
Finally, using Plancherel's formula for the Dunkl transform, the desired inequality follows.

\fd

\medskip

The following result presents further $L^p_k(\R^d)$-inequalities involving the fractional Dunkl Laplacian.

\begin{Prop} Let $0<\theta<d+2\gamma$ and $\alpha\geq\theta$.
\begin{description}
\item[i)] There exists  a positive constant $C=C(k,d,\theta)$ such that
$$
\left\|(-\Delta_k)^{\frac{\alpha-\theta}{2}}f\right\|_{\infty}\leq C\left(\left\|(-\Delta_k)^{\frac{\alpha}{2}}f\right\|_{\infty}+ \left\|(-\Delta_k)^{\frac{\alpha}{2}}f\right\|_{L^1_k(\R^d)}\right),\quad \forall\ f\in \mathcal{S}(\R^d).
$$
\item[ii)] Let  $1<p<\frac{d+2\gamma}{\theta}$. Then there exists a positive constant $C=C(k,d,\theta,p)$ such that
\begin{equation}\label{Sobolev inequality Frac Lap}
\left\|(-\Delta_k)^{\frac{\alpha-\theta}{2}}f\right\|_{L^{\frac{p(d+2\gamma)}{d+2\gamma-\theta p}}_k(\R^d)}\leq C \left\|(-\Delta_k)^{\frac{\alpha}{2}}f\right\|_{L^p_k(\R^d)},\quad \forall\ f\in \mathcal{S}(\R^d).
\end{equation}
\end{description}
\end{Prop}
\emph{Proof:} \textbf{i)} Applying (\ref{inequality fund solution}) to the function $(-\Delta_k)^{\frac{\alpha}{2}}f$ and using (\ref{frac lap-Riesz poten}),  we deduce the desired estimate.\\
\textbf{ii)}  First, note that from (\ref{norm frac Lapl Lp}), the right hand side of (\ref{Sobolev inequality Frac Lap}) is finite. The stated inequality then follows directly from   (\ref{frac lap-Riesz poten}) and the boundedness of the $\Delta_k$-Riesz potential
$$
\mathcal{I}_{k,\theta}:L^p_k(\R^d) \longrightarrow L^{\frac{p(d+2\gamma)}{d+2\gamma-\theta p}}(\R^d)
$$
as established in  \cite{GalRejSifi, Hassani}.

\fd

\section{Extension problem for the fractional Dunkl Laplacian}
 The goal of this section is to study the $\Delta_k$-extension problem by using a fractional Poisson kernel and to show  the Caffarelli-Silvestre type  harmonic
characterization of the fractional Dunkl Laplacian.
\subsection{$\Delta_k$-Fractional Poisson kernel}
Following \cite{Caffarelli, Stinga thesis}, we  introduce the fractional Poisson kernel of index $\alpha>0$  in terms the $\Delta_k$-heat kernel (\ref{heat kernel 1 Intr}) by
\begin{equation}\label{frac Poisson kernel}
\mathcal{P}_{k,\alpha}(t,x,y):=\frac{t^{2\alpha}}{4^{\alpha}\Gamma(\alpha)}\int_0^{\infty}p_k(s,x,y)
e^{-\frac{t^2}{4s}}\frac{ds}{s^{\alpha+1}},\quad \quad\quad t>0,\quad  x,y\in\R^d.
\end{equation}
The corresponding $\alpha$-Poisson integral transform is
\begin{equation}\label{frac Poisson integral}
\mathcal{P}_{k,\alpha}[f](t,x):=\int_{\R^d}\mathcal{P}_{k,\alpha}(t,x,y)f(y)\omega_k(y)dy.
 \end{equation}
The kernel  $ \mathcal{P}_{k,\alpha}(t,\cdot,\cdot)$, $t>0$, inherits some properties from the Dunkl heat kernel:
\begin{itemize}[labelsep=0.5em,left=0em]
\item It is positive and symmetric.
\item By virtue of the equality
\begin{equation}\label{probability measure}
\frac{t^{2\alpha}}{4^{\alpha}\Gamma(\alpha)}\int_0^{\infty}e^{-\frac{t^2}{4s}}\frac{ds}{s^{\alpha+1}}=1,\quad \quad \quad \forall\ t>0,
\end{equation}
it follows that  $\|\mathcal{P}_{k,\alpha}(t,x,\cdot)\|_{L^1_k(\R^d)}=1$, for all $x\in\R^d$. Furthermore, from the property of the Dunkl transform of the $\Delta_k$-heat kernel (\ref{Dunkl trans heat}), the Dunkl transform of the function $\mathcal{P}_{k,\alpha}(t,x,\cdot)$ is given by
\begin{equation}\label{EP-PoissonK-DT}
\begin{split}
\forall\ \xi\in \R^d,\quad \quad \mathcal{F}_k\big(\mathcal{P}_{k,\alpha}(t,x,\cdot)\big)(\xi)&=\mathcal{P}_{k,\alpha}\big[E_k(-i\xi,\cdot)\big](t,x)\\
&=E_k(-ix,\xi)\frac{t^{2\alpha}}{4^{\alpha}\Gamma(\alpha)}\int_0^{\infty}e^{-s\|\xi\|^2}e^{-\frac{t^2}{4s}}\frac{ds}{s^{\alpha+1}}.
\end{split}
\end{equation}
\item For each $x$, the function $\mathcal{P}_{k,\alpha}(t,x,\cdot)$ lies in  $L^2_k(\R^d)$. Hence, using the action of the Dunkl translation on $L^2_k(\R^d)$-functions
    (\ref{translation L2}), along with (\ref{EP-PoissonK-DT}), we can  write
\begin{equation}\label{frac Poisson kernel}
\mathcal{P}_{k,\alpha}(t,x,y)=\tau_{-x}\mathcal{P}_{k,\alpha}(t,0,\cdot)(y).
\end{equation}
In particular, the $\alpha$-fractional Poisson integral of $f\in L^2_k(\R^d)$ can be expressed as a Dunkl convolution product:
\begin{equation}\label{Poisson integral convolution}
\mathcal{P}_{k,\alpha}[f](t,\cdot)=\mathcal{P}_{k,\alpha}(t,0,\cdot)\ast_kf.
\end{equation}
\end{itemize}
\medskip

The following result about the fractional Poisson kernel  is of particular importance on itself.
\begin{Prop} For $0<\theta<d+2\gamma$ and $t>0$, define  the $t$-Dunkl-Riesz  kernel of index $\theta$ by
\begin{equation}\label{t-Riesz kernel}
R_{k,\theta}(t,x,y)=\frac{1}{\Gamma(\theta/2)}\int_0^{\infty}p_k(s,x,y)e^{-\frac{t^2}{4s}}s^{\frac{\theta}{2}-1}ds, \quad x,y\in\R^d.
\end{equation}
If $\theta<\frac{d+2\gamma}{2}$, then for each $x\in \R^d$, the function $R_{k,\theta}(t,x,\cdot)$ belongs to the Sobolev space $H_k^{\theta}(\R^d)$, and
\begin{equation}\label{poisson kernel-Riesz kernel}
\mathcal{P}_{k,\frac{\theta}{2}}(t,x,\cdot)=(-\Delta_k)^{\theta/2}R_{k,\theta}(t,x,\cdot), \quad \text{in}\quad L^2_k(\R^d).
\end{equation}
\end{Prop}

Note that the kernel $R_{k,\theta}(t,x,y)$   is positive, symmetric in $x$ and $y$,  and it  approaches the $\Delta_k$-Riesz kernel (\ref{Riesz kernel def})-(\ref{Riesz kernel def1}) as $t\to0$:
$$
\forall\ x,y\in \R^d, \quad \lim_{t\to0}R_{k,\theta}(t,x,y)=R_{k,\theta}(x,y).
$$
\emph{Proof:} From expressions (\ref{t-Riesz kernel}) and (\ref{heat kernel 1 Intr}) for the kernels $R_{k,\theta}(t,x,y)$ and $p_{k}(s,x,y)$, we can see that
$$
R_{k,\theta}(t,0,y)=A_{k,d,\theta}\big(t^2+\|y\|^2\big)^{\frac{\theta-d-2\gamma}{2}},
$$
where $A_{k,d,\theta}$ is the constant (\ref{kappa}).
In particular, using spherical coordinates we see that the function $R_{k,\theta}(t,0,\cdots)$ belongs to $L^2_k(\R^d)$  whenever $d+2\gamma>2\theta$. Moreover, since the function $R_{k,\theta}(t,0,\cdot)$ is radial and of class $C^{\infty}$ on $\R^d$, the above formula together with (\ref{P6}) and Fubini's theorem  imply that $R_{k,\theta}(t,x,\cdot)=\tau_{-x}R_{k,\theta}(t,0,\cdot)$. Therefore, under the assumption $d+2\gamma>2\theta$, we conclude that $R_{k,\theta}(t, x,\cdot)\in L^2_k(\R^d)$ for all $x\in\R^d$. \\
We now claim that
\begin{equation}\label{t-Riesz kernel Fourier}
\mathcal{F}_k\big(R_{k,\theta}(t,x,\cdot))(\xi)=\frac{E_k(-ix,\xi)}{\Gamma(\theta/2)}\int_0^{\infty}e^{-s\|\xi\|^2}e^{-\frac{t^2}{4s}}s^{\frac{\theta}{2}-1}ds\quad\quad \text{in}\quad L^2_k(\R^d).
\end{equation}
Indeed, applying  the Cauchy-Schwarz inequality and Plancherel's formula for the Dunkl transform, we deduce for all $t>0$ and all $f\in L^2_k(\R^d)$ that
\begin{align*}
\frac{1}{\Gamma(\theta/2)}\int_0^{\infty}e^{s\Delta_k}|f|(x)e^{-\frac{t^2}{4s}}s^{\frac{\theta}{2}-1}ds
&\leq\frac{\|f\|_{L^2_k(\R^d)}}{\Gamma(\theta/2)}\int_0^{\infty}\|e^{-s\|.\|^2}\|_{L^2_k(\R^d)}e^{-\frac{t^2}{4s}}s^{\frac{\theta}{2}-1}ds\\
&\simeq \|f\|_{L^2_k(\R^d)}\int_0^{\infty}e^{-\frac{t^2}{4s}}s^{\frac{\theta-d-2\gamma}{2}-1}ds<+\infty.
\end{align*}
Thus, Fubini's theorem applies and  implies that for $\theta\in (0,d+2\gamma)$ and $f\in L^2_k(\R^d)$:
\begin{equation}\label{t-potential heat sgp}
 {\prs{R_{k,\theta}(t,x,\cdot)}{f}}_{L^2_k(\R^d)}=\frac{1}{\Gamma(\theta/2)}\int_0^{\infty}e^{s\Delta_k}f(x)e^{-\frac{t^2}{4s}}s^{\frac{\theta}{2}-1}ds.
\end{equation}
Hence, applying  Plancherel's formula twice and using (\ref{t-potential heat sgp}), we obtain
\begin{align*}
{\prs{\mathcal{F}_k\big(R_{k,\theta}(t,x,\cdot)\big)}{\mathcal{F}_k(f)}}_{L^2_k(\R^d)}&=c_k^{-2}{\prs{R_{k,\theta}(t,x,\cdot)}{f}}_{L^2_k(\R^d)}\\
&=\frac{1}{\Gamma(\theta/2)}\int_0^{\infty}{\prs{\mathcal{F}_k\big(p_k(s,x,\cdot)\big)}{\mathcal{F}_k(f)}}_{L^2_k(\R^d)}e^{-\frac{t^2}{4s}}s^{\frac{\theta}{2}-1}ds\\
&=\int_{\R^d}\left(\frac{E_k(-ix,\xi)}{\Gamma(\theta/2)}\int_0^{\infty}e^{-s\|\xi\|^2}e^{-\frac{t^2}{4s}}s^{\frac{\theta}{2}-1}ds\right)\mathcal{F}_k(f)(\xi)\omega_k(\xi)d\xi.
\end{align*}
This shows the claimed formula (\ref{t-Riesz kernel Fourier}). \\
Now, by substituting (\ref{t-Riesz kernel Fourier}) and  making the change of variables $s'=\frac{t^2}{4s\|\xi\|^2}$ into relation   (\ref{EP-PoissonK-DT}), we conclude that
$$
\mathcal{F}_k\big(\mathcal{P}_{k,\frac{\theta}{2}}(t,x,\cdot)\big)=\|\cdot\|^{\theta}\mathcal{F}_k\big(R_{k,\theta}(t,x,\cdot)\big).
$$
Thus, $R_{k,\theta}(t,x,\cdot)\in H_k^{\theta}(\R^d)$ and the stated formula (\ref{poisson kernel-Riesz kernel}) follows.

\fd

\begin{Remark} Taking $x=0$ in (\ref{poisson kernel-Riesz kernel}) and  substituting the explicit form of the constant $A_{k,d,\alpha}$ (\ref{kappa}), we deduce that identity (\ref{poisson kernel-Riesz kernel}) can be rewritten as
$$
(-\Delta_k)^{\frac{\theta}{2}}\left(t^2+\|\cdot\|^2\right)^{-\frac{d+2\gamma-\theta}{2}}=2^{\theta}\frac{\Gamma(\frac{d+2\gamma+\theta}{2})}{\Gamma(\frac{d+2\gamma-\theta}{2})}
t^{\theta}\left(t^2+\|\cdot\|^2\right)^{-\frac{d+2\gamma+\theta}{2}}\quad \quad \text{in}\quad L^2_k(\R^d).
$$
\end{Remark}

\begin{Cor} Let $0<\theta<\frac{d+2\gamma}{2}$. Then the $\frac{\theta}{2}$-fractional Poisson integral of any $f\in H_k^{\theta}(\R^d)$  takes the following forms
\begin{align}
\mathcal{P}_{k,\frac{\theta}{2}}[f](t,x)&\label{Poissin integral-t-pot}=\mathcal{I}_{k,\theta}[(-\Delta_k)^{\theta/2}f](t,x)\\
\label{Extension problem solution2}&=
\frac{1}{\Gamma(\frac{\theta}{2})}\int_0^{\infty}e^{s\Delta_k}(-\Delta_k)^{\theta/2}f(x)e^{-\frac{t^2}{4s}}s^{\frac{\theta}{2}-1}ds.
\end{align}
Here $\mathcal{I}_{k,\theta}[f](t,\cdot)$ is the $(t,\theta)$-Dunkl-Riesz potential of $f$ defined by
\begin{equation}\label{t-potential}
\mathcal{I}_{k,\theta}[f](t,x):=\int_{\R^d}R_{k,\theta}(t,x,y)f(y)\omega_k(y)dy.
\end{equation}
\end{Cor}
\noindent \emph{Proof:} Fix $f\in H_k^{\theta}(\R^d)$. From (\ref{t-potential heat sgp}), we get
$$
\mathcal{I}_{k,\theta}[(-\Delta_k)^{\theta/2}f](t,x)
=\frac{1}{\Gamma(\theta/2)}\int_0^{\infty}e^{s\Delta_k}(-\Delta_k)^{\theta/2}f(x)e^{-\frac{t^2}{4s}}s^{\frac{\theta}{2}-1}ds.
$$
Hence, relation (\ref{Poissin integral-t-pot}) is a direct consequence of (\ref{poisson kernel-Riesz kernel}) and the symmetry property of the fractional Dunkl Laplacian. Moreover, combining  (\ref{t-potential heat sgp}) and (\ref{Poissin integral-t-pot}), we obtain  (\ref{Extension problem solution2}).

\fd
\begin{Remark}
\emph{The usual counterpart of formula (\ref{Extension problem solution2}) was obtained by Stinga \cite{Stinga user} using a different approach.}
\end{Remark}

\subsection{$\Delta_k$-extension problem}
The following  result  states a Caffarelli-Silvestre type relation for the fractional Dunkl Laplacian:
\begin{Thm}\label{Thm ExtP} Let $0<\alpha<1$ and $f\in L^2_k(\R^d)$ be fixed.
\begin{description}[labelsep=0.5em,left=0em]
\item[1)] The function $(t,x)\mapsto \mathcal{P}_{k,\alpha}[f](t,x)$ from  (\ref{frac Poisson integral}) satisfies  the $\Delta_k$-extension problem:
\begin{equation}\label{extension problem}
\begin{cases}
\partial_t^2u(t,x)+\frac{1-2\alpha}{t}\partial_tu(t,x)+\Delta_{k}u(t,x)=0, & \hbox{in}\ (0,+\infty)\times\R^d \\
\lim_{t\to0}u(t,\cdot)=f, & \hbox{in}\quad  L^2_k(\R^d).
\end{cases}
\end{equation}
\item[2)] If $f\in H_k^{2\alpha}(\R^d)$, then the following equalities hold in $L^2_k(\R^d)$:
\begin{align}
\begin{split}\label{Frac Lap-EP 2}
(-\Delta_k)^{\alpha}f(x)&=\lim_{t\to0}\frac{4^{\alpha}\Gamma(\alpha)}{\Gamma(-\alpha)}
t^{-2\alpha}\big(\mathcal{P}_{k,\alpha}[f](t,x)-f(x)\big)
\end{split}\\
\begin{split}\label{Frac Lap-Extension problem 1}
&=\lim_{t\to0}\frac{4^{\alpha}\Gamma(\alpha)}{2\alpha \Gamma(-\alpha)}t^{1-2\alpha}
\partial_t\mathcal{P}_{k,\alpha}[f](t,x)
\end{split}\\
\begin{split}\label{CS new form}
&=\lim_{t\to0}\frac{4^{\alpha}\Gamma(\alpha)}{2\alpha \Gamma(-\alpha)}t^{1-2\alpha}
\partial_tI_{k,2\alpha}[(-\Delta_k)^{\alpha}f](t,x),
\end{split}
\end{align}
where $\mathcal{I}_{k,\theta}[f](t,x)$ is the $(t,\theta)$-Dunkl-Riesz integral transform of $f$ defined in (\ref{t-potential}).
\end{description}
\end{Thm}
\emph{Proof:}
 \textbf{1)} Let $s>0$. By virtue of Dunkl convolution property (\ref{Dunkl trans convolution}) and the fact that  $f\in L^2_k(\R^d)$, we have
\begin{equation*}\label{heat semigp L2}
e^{s\Delta_k}f(x)=p_k(s,0,.)\ast_k f(x)=c_k^{-2}\int_{\R^d}e^{-s\|\xi\|^2}\mathcal{F}_k(f)(\xi)E_k(-ix,\xi)\omega_k(\xi)d\xi.
\end{equation*}
Then,  by using  estimate (\ref{Dunkl kernel inequality2}) for the Dunkl kernel and  differentiating under the  integral sign, we conclude that $e^{s\Delta_k}f$ is of class  $C^{\infty}$ on $\R^d$. Moreover, the Cauchy-Schwarz inequality and Plancherel's formula for the Dunkl transform $\mathcal{F}_k$ imply that
\begin{align*}
\forall\ s>0,\ \forall\ x\in \R^d,\quad \quad |e^{s\Delta_k}f(x)|\lesssim s^{-\frac{d}{2}-\gamma}\|f\|_{L^2_k(\R^d)}.
\end{align*}
Therefore, we can apply Fubini's theorem to express the $\alpha$-Poisson integral of $f$ (\ref{frac Poisson integral}) as
$$
\mathcal{P}_{k,\alpha}[f](t,x)=\frac{t^{2\alpha}}{4^{\alpha}\Gamma(\alpha)}
\int_0^{\infty}e^{s\Delta_k}f(x)e^{-\frac{t^2}{4s}}\frac{ds}{s^{1+\alpha}}.
$$
From this formula, we see that the function $\mathcal{P}_{k,\alpha}[f](t,\cdot)$ is of class $C^{\infty}$ on $\R^d$. In addition, the heat equation  $\Delta_ke^{s\Delta_k}f=\partial_se^{s\Delta_k}f$,  along with  an integration by parts,   yields
\begin{align*}
\Delta_k\mathcal{P}_{k,\alpha}[f](t,x)
&=\frac{4\alpha(\alpha+1)}{t^2}\Big(\mathcal{P}_{k,\alpha+1}[f](t,x)-\mathcal{P}_{k,\alpha+2}[f](t,x)\Big).
\end{align*}
Similarly, for each fixed $x\in\R^d$, the function $t\mapsto \mathcal{P}_{k,\alpha}[f](t,x)$ is of class $C^{\infty}$ on $(0,+\infty)$, and
\begin{equation}\label{Poisson Integral der t}
\partial_t\mathcal{P}_{k,\alpha}[f](t,x)=\frac{2\alpha}{t}\Big(\mathcal{P}_{k,\alpha}[f](t,x)-\mathcal{P}_{k,\alpha+1}[f](t,x)\Big).
\end{equation}
A second differentiation in $t$ leads to
\begin{align*}
\partial_t^2\mathcal{P}_{k,\alpha}[f](t,x)=\frac{2\alpha}{t^2}\Big((2\alpha-1)\mathcal{P}_{k,\alpha}[f](t,x)-
(4\alpha+1)\mathcal{P}_{k,\alpha+1}[f](t,x)+2(\alpha+1)\mathcal{P}_{k,\alpha+2}[f](t,x)\Big).
\end{align*}
These relations establish  that the function $\mathcal{P}_{k,\alpha}[f](t,x)$ satisfies the partial differential equation in (\ref{extension problem}).\\
On the other hand, in view of relation (\ref{probability measure}) and the Jensen inequality, we have
\begin{align*}
\|\mathcal{P}_{k,\alpha}[f](t,\cdot)-f\|^2_{L^2_k(\R^d)}&\leq \frac{t^{2\alpha}}{4^{\alpha}\Gamma(\alpha)}\int_0^{\infty} \|e^{s\Delta_k}f-f\|^2_{L^2_k(\R^d)}e^{-\frac{t^2}{4s}}\frac{ds}{s^{\alpha+1}}\\
&= \frac{1}{\Gamma(\alpha)}\int_0^{\infty} \|e^{\frac{t^2}{4r}\Delta_k}f-f\|^2_{L^2_k(\R^d)}e^{-r}r^{\alpha-1}dr.
\end{align*}
Accordingly, the desired equality $\lim_{t\to 0}\mathcal{P}_{k,\alpha}[f](t,\cdot)=f$ in $L^2_k(\R^d)$ follows the following property of the Dunkl heat semigroup \cite{Rosler4}
$$
\lim_{t\to 0}\|e^{\frac{t^2}{4\theta}\Delta_k}f-f\|^2_{L^2_k(\R^d)}=0.
$$
\noindent\textbf{2)} First,  observe that the $\alpha$-Poisson integral transform  $\mathcal{P}_{k,\alpha}[\cdot](t,\cdot)$ leaves the space $L^2_k(\R^d)$ invariant, with
$$
\left\|\mathcal{P}_{k,\alpha}[f](t,\cdot)\right\|_{L^2_k(\R^d)}\leq \|f\|_{L^2_k(\R^d)},\quad \quad \forall\ f\in L^2_k(\R^d).
$$
Therefore, using expression (\ref{Poisson integral convolution}), formula(\ref{EP-PoissonK-DT}) for $\mathcal{F}_k(\mathcal{P}_{k,\alpha}(t,x,\cdot))$  and  (\ref{probability measure}) together with the properties of the Dunkl transform, we obtain
\begin{align}\label{good R}
\mathcal{F}_k\big(\mathcal{P}_{k,\alpha}[f](t,\cdot)-f\big)&=\mathcal{F}_k(f)\frac{t^{2\alpha}}{4^{\alpha}\Gamma(\alpha)}
\int_0^{\infty}\big(e^{-s\|\cdot\|^2}-1\big)e^{-\frac{t^2}{4s}}\frac{ds}{s^{1+\alpha}}\quad \text{in} \quad L^2_k(\R^d).
\end{align}
Now, taking into account the fact that $f\in H_k^{2\alpha}(\R^d)$, applying the dominated convergence theorem and using relations (\ref{generateur infinitesimal L2}) and (\ref{lambda puissance}), we conclude that
$$
\lim_{t\to0}t^{-2\alpha}\mathcal{F}_k\big(\mathcal{P}_{k,\alpha}[f](t,\cdot)-f\big)=\frac{\Gamma(-\alpha)}{4^{\alpha}\Gamma(\alpha)}\|\cdot\|^{2\alpha}\mathcal{F}_k(f)
=\frac{\Gamma(-\alpha)}{4^{\alpha}\Gamma(\alpha)}\mathcal{F}_k\big((-\Delta_k)^{\alpha}f\big)\quad \text{in} \quad L^2_k(\R^d).
$$
Finally, using Plancherel's theorem for the Dunkl transform, we obtain (\ref{Frac Lap-EP 2}). \\
Next, using the  time derivative formula (\ref{Poisson Integral der t}), identity (\ref{good R}) and integrating  by parts and recalling that $\Gamma(-\alpha):=-\Gamma(1-\alpha)/\alpha$, we arrive at the following equalities in $L^2_k(\R^d)$:
\begin{align*}
\lim_{t\to0}t^{1-2\alpha}\mathcal{F}_k\big(\partial_t\mathcal{P}_{k,\alpha}[f](t,\cdot)\big)&=2\alpha \lim_{t\to0}t^{-2\alpha}\Big(\mathcal{F}_k\big(\mathcal{P}_{k,\alpha}[f](t,\cdot)-\mathcal{F}_k\big(\mathcal{P}_{k,\alpha+1}[f](t,\cdot)\big)\Big)\\
&=2\alpha \lim_{t\to0}t^{-2\alpha}\Big(\mathcal{F}_k\big(\mathcal{P}_{k,\alpha}[f](t,\cdot)-f\big)+\mathcal{F}_k\big(f-\mathcal{P}_{k,\alpha+1}[f](t,\cdot)\big)\Big)\\
&=-2\alpha \mathcal{F}_k(f)\lim_{t\to0}\frac{\|\xi\|^2}{4^{\alpha}\Gamma(\alpha+1)}\int_0^{\infty}e^{-\frac{t^2}{4s}}e^{-s\|\xi\|^2}\frac{ds}{s^{\alpha}}\\
&=2\alpha\frac{\Gamma(-\alpha)}{4^{\alpha}\Gamma(\alpha)}\|\cdot\|^{2\alpha}\mathcal{F}_k(f),
\end{align*}
which establishes (\ref{Frac Lap-Extension problem 1}).
Finally, we obtain (\ref{CS new form}) from (\ref{Poissin integral-t-pot}) and  the relationship (\ref{Frac Lap-Extension problem 1}), which states that  $\mathcal{P}_{k,\alpha}[\cdot]=I_{k,2\alpha}(-\Delta_k)^{\alpha}$.

\fd

The following result shows that the Dunkl kernel $E_k(i\xi,\cdot)$, which is bounded and smooth on $\R^d$, also satisfies the Caffarelli-Silvestre type relation (\ref{Frac Lap-EP 2})
\begin{Prop}
For any $x, \xi\in\R^d$, we have
\begin{equation}\label{frac lap Dunkl kernel}
\begin{split}
(-\Delta_k)^{\alpha}\big(E_k(i\xi,\cdot)\big)(x)&=\|\xi\|^{2\alpha}E_k(i\xi,x)\\
&=\lim_{t\to0}\frac{4^{\alpha}\Gamma(\alpha)}{\Gamma(-\alpha)}
t^{-2\alpha}\big(\mathcal{P}_{k,\alpha}\big[E_k(i\xi,\cdot)](t,x)-E_k(i\xi,x)\big).
\end{split}
\end{equation}
\end{Prop}
\emph{Proof:}  According to relation  (\ref{EP-PoissonK-DT}),   the $\alpha$-Poisson transform of the function $E_k(i\xi,\cdot)$ is given by
\begin{align*}
\mathcal{P}_{k,\alpha}[E_k(i\xi,\cdot)](t,x)=E_k(i\xi,x)\frac{t^{2\alpha}}{4^{\alpha}\Gamma(\alpha)}\int_0^{\infty}e^{-s\|\xi\|^2}e^{-t^2/{4s}}\frac{ds}{s^{\alpha+1}}.
\end{align*}
Therefore
\begin{align*}
&\lim_{t\to0}\frac{4^{\alpha}\Gamma(\alpha)}{\Gamma(-\alpha)}t^{-2\alpha}\Big(\mathcal{P}_{k,\alpha}[E_k(i\xi,\cdot)](t,x)-E_k(i\xi,x)\Big)\\
&=\lim_{t\to0}\frac{E_k(i\xi,x)}{\Gamma(-\alpha)}\int_0^{\infty}\big(e^{-s\|\xi\|^2}-1\big)e^{-t^2/{4s}}\frac{ds}{s^{\alpha+1}}\\
&=\|\xi\|^{2\alpha}E_k(i\xi,x)\\
&=(-\Delta_k)^{\alpha}\big(E_k(i\xi,\cdot)\big)(x),
\end{align*}
where we have used  the monotone convergence theorem and (\ref{lambda puissance}) in the third equality, and relation (\ref{Dunkl kerenl-Lap}) in the last equality.

\fd


\appendix
 \section{Appendix}
\subsection{Dunkl kernel and Dunkl transform}
$\bullet$ The Dunkl kernel $E_k$ is a positive and symmetric kernel on $\R^d\times\R^d$, defined by
\begin{equation}\label{Dunkl kernel def}
E_k(x,y):=V_k(e^{\prs{\cdot}{y}})(x)=\int_{\R^d}e^{\prs{z}{y}}d\mu_x(z).
\end{equation}
For $y\in \R^d$, $E_k(\cdot,y)$ is the unique solution of the system
$$
\forall\ \xi\in\R^d, \quad D_{\xi}f(x)=\prs{\xi}{y}f(x)\quad \text{and}\quad f(0)=1.
$$
The Dunkl kernel extends to an analytic function on $\C^d\times\C^d$.  Furthermore, it satisfies the following properties \cite{Dunkl2, dunklxu, Dejeu}
\begin{enumerate}[labelsep=0.5em,left=0em]
\item  For all $\lambda\in \C$, $x,y\in \C^d$ and all $g\in W$, we have
$E_k(\lambda x,y)=E_k(x,\lambda y)$ and $E_k(gx,gy)=E_k(x,y)$.
\item \emph{Boundedness on the imaginary axis:} For all $x,y\in\R^d$, we have
\begin{equation}\label{Dunkl kernel inequality1}
|E_k(-ix,y)|\leq 1.
\end{equation}
\item \emph{Estimates for partial derivatives:} For all $x\in \R^d$, all $y\in\C^d$ and all multi-indices $\nu\in\N^d$,
\begin{equation}\label{Dunkl kernel inequality2}
\big|\partial_y^{\nu}E_k(x,y)\big|\leq \|x\|^{|\nu|}\max_{g\in W}e^{Re\prs{gx}{y}}.
\end{equation}
\end{enumerate}
$\bullet$ The Dunkl transform of a function $f\in L^1_k(\R^d)$ is defined by \cite{Dejeu, Rosler4}
\begin{equation}\label{Dunkltransf}
\mathcal{F}_k(f)(\xi):=\int_{\R^d}f(x)E_k(-ix,\xi)\omega_k(x)dx,\quad \xi\in\R^d.
\end{equation}
 The Dunkl transform shares many properties with the usual Euclidean Fourier transform \cite{Dunkl3, Dejeu, Rosler4}. In particular,
\begin{enumerate}[labelsep=0.5em,left=0em]
\item For $f\in \mathcal{S}(\R^d)$ and  $\xi\in\R^d$, we have
\begin{equation*}\label{derivation dunkl transf}
D_{\xi}\mathcal{F}_k(f)=-i\mathcal{F}_k(\prs{\xi}{.}f), \quad  \mathcal{F}_k(D_{\xi} f)=i\prs{\xi}{.}\mathcal{F}_k(f).
\end{equation*}
\item \emph{Isomorphism on the Schwartz space:} $\mathcal{F}_k$ is an isomorphism of $\mathcal{S}(\R^d)$ onto itself and its inverse is given by
\begin{equation}\label{Dunkltransfinverse}
\mathcal{F}_k^{-1}(f)(x)=c_k^{-1}\mathcal{F}_k(f)(-x), x\in\R^d.
\end{equation}
where $c_k$ is the Mehta constant (\ref{Mehta}).
\item \emph{Preservation of radiality:} The Dunkl transform of a radial function is again a radial function.
\item \emph{Injectivity:} The operator $\mathcal{F}_k : L^1_k(\R^d)\longrightarrow \mathcal{C}_0(\R^d)$ is injective.
\item \emph{Inversion formula:} If $f$ and $\mathcal{F}_k(f)$ are in $L^1_k(\R^d)$, then
$
f(x)=c_k^{-2}\mathcal{F}_k\big(\mathcal{F}_k(f)\big)(-x).
$
\item \emph{Plancherel theorem:} $c_k^{-1}\mathcal{F}_k$ extends to an isometric automorphism of $L^2_k(\R^d)$.
\item The Dunkl transform is a topological isomorphism of the space $\mathcal{S}'(\R^d)$ onto itself. For $T\in \mathcal{S}'(\R^d)$, its Dunkl transform is defined by
$$
\prs{\mathcal{F}_k(T)}{\varphi}:=\prs{T}{\mathcal{F}_k(\varphi)}, \quad \varphi\in \mathcal{S}(\R^d).
$$
\end{enumerate}

\subsection{Dunkl translation operators and Dunkl convolution}
By means of the Dunkl intertwining operator and its inverse, the Dunkl translation operators $\tau_x, x\in \R^d$, are defined
on $\mathcal{C}^{\infty}(\R^d)$ by \cite{Trimeche2}
\begin{equation}\label{translation Cinfty}
\forall\ y\in\R^d,\quad \tau_xf(y)=\int_{\R^d}V_k\circ T_z\circ V_k^{-1}(f)(y)d\mu_x(z),
\end{equation}
where $T_x$ is the classical translation operator given by $T_xf(y)=f(x+y)$.\\
The operators $\tau_x$, $x\in \R^d$, satisfy the following properties:
\begin{enumerate}[labelsep=0.5em,left=0em]
\item \emph{Product formula:} for all $x,y,\xi\in \R^d$,
\begin{equation}\label{product formula}
\tau_xE_k(\xi,\cdot)(y)=E_k(\xi,x)E_k(\xi,y).
\end{equation}
\item \emph{Continuity:} For all $x\in\R^d$, the operator $\tau_x$ is continuous from  $\mathcal{C}^{\infty}(\R^d)$ into itself.
\item For all $f\in \mathcal{C}^{\infty}(\R^d)$ and all $x,y\in \R^d$, we have
$\tau_xf(0)=f(x)$ and $\tau_xf(y)=\tau_yf(x)$.
\item  \emph{Commutation with Dunkl operators:}  For all $f\in \mathcal{C}^{\infty}(\R^d)$ and all $x,\xi\in \R^d$: $\tau_x(D_{\xi}f)=D_{\xi}(\tau_xf)$.
\item \emph{Action on radial functions:} If $f=\widetilde{f}(\|\cdot\|)\in \mathcal{C}^{\infty}(\R^d)$ is radial, then according to \cite{Rosler2} we have
  \begin{equation}\label{P6}
\forall\ x\in\R^d,\quad \tau_xf(y)=\int_{\R^d}\widetilde{f}(\sqrt{\|x\|^2+\|y\|^2+2\prs{x}{z}})d\mu_y(z).
\end{equation}
\item \emph{Duality type relation:} Let  $f\in \mathcal{C}^{\infty}(\R^d)$ and $\phi\in \mathcal{D}(\R^d)$.  According to \cite{GalRej} (Proposition 2.1), the following holds:
\begin{equation}\label{Duality-transl}
\int_{\R^d}\tau_xf(y)\phi(y)\omega_k(y)dy=\int_{\R^d}f(y)\tau_{-x}\phi(y)\omega_k(y)dy,\quad \forall\ x\in \R^d.
\end{equation}
\end{enumerate}
On the other hand, when $f\in L^2_k(\R^d)$, $\tau_xf$ is the $L^2_k(\R^d)$-function defined as a Dunkl transform multiplier:
\begin{equation}\label{translation L2}
\mathcal{F}_k(\tau_xf)=E_k(ix,\cdot)\mathcal{F}_k(f).
\end{equation}
If $f\in\mathcal{S}(\R^d)$, then the formulas (\ref{translation Cinfty}) and (\ref{translation L2}) coincide,
$\tau_xf\in\mathcal{S}(\R^d)$ and we have \cite{Trimeche2}:
\begin{equation}\label{translation Schwartz}
\tau_xf(y)=c_k^{-2}\int_{\R^d}\mathcal{F}_k(f)(\xi)E_k(ix,\xi)E_k(iy,\xi)\omega_k(\xi)d\xi,\quad y\in\R^d.
\end{equation}
Notice that this formula remains  valid  whenever the two functions $f$ and $\mathcal{F}_k(f)$ are both in  $L^1_k(\R^d)$.\\
Classically, the Dunkl convolution product of two $L^2_k(\R^d)$-functions $f$ and $g$ is defined by
$$
f\ast_kg(x)=\int_{\R^d}f(y)\tau_xg(-y)\omega_k(y)dy.
$$
Using the Dunkl transform, we have
\begin{equation}\label{Dunkl trans convolution}
f\ast_kg(x)=c_k^{-2}\int_{\R^d}\mathcal{F}_k(f)(\xi)\mathcal{F}_k(g)(\xi)E_k(ix,\xi)\omega_k(\xi)d\xi.
\end{equation}

\subsection{Dunkl heat kernel}
 The Dunkl heat semigroup  is defined  by \cite{Rosler3,  Rosler4}
\begin{equation}\label{heat semigroup def}
e^{t\Delta_k}f(x)=p_k(t,0,\cdot)\ast_kf(x)=\int_{\R^d}p_k(t,x,y)f(y)\omega_k(y)dy, \quad t>0, \ f\in \mathcal{S}(\R^d),
\end{equation}
where $p_k(t,x,y)$ is the $\Delta_k$-heat kernel defined by (\ref{heat kernel 1 Intr}).
This kernel satisfies the following properties:
\begin{enumerate}[labelsep=0.5em,left=0em]
\item \emph{Symmetry:}  For each $t>0$ and $x,y\in \R^d$,  $p_k(t,x,y)=p_k(t,y,x)$.
\item \emph{$L^1_k(\R^d)$-Norm and Dunkl Transform:} For every $t>0$ and $x\in \R^d$, $\|p_k(t,x,\cdot)\|_{L^1_k(\R^d)}=1$ and
\begin{equation}\label{Dunkl trans heat}
\mathcal{F}_k(p_k(t,x,\cdot))(\xi)=e^{-t\|\xi\|^2}E_k(-ix,\xi),\quad\quad  \forall\ \xi\in\R^d.
\end{equation}
In particular, for all $t>0$ and all $x,y\in \R^d$,
\begin{align}
p_k(t,x,y) &=\label{heat Fourier2}c_k^{-2}\int_{\R^d}e^{-t\|\xi\|^2}E_k(-ix,\xi)E_k(iy,\xi)\omega_k(\xi)d\xi.
\end{align}
\item \emph{Semigroup Property:} For all $t,s>0$ and all $x,y\in\R^d$,
\begin{equation}\label{heat semi group prop}
p_k(t+s,x,y)=\int_{\R^d}p_k(t,x,z)p_k(s,y,z)\omega_k(z)dz.
\end{equation}
\end{enumerate}

\bigskip

\bigskip

\noindent \textbf{Conflict of Interest Statement:}  No conflict of interest was reported by the author.

\bigskip

\noindent\textbf{Data Availability Statement:}  The manuscript has no associated data.

\bigskip

\noindent\textbf{Funding Declaration:} The author declares that no funding was received for this research.

\end{document}